\documentclass[11pt]{article}

\usepackage{amsfonts,amsmath,latexsym,color, epsfig, amssymb, url, booktabs, tikz,
setspace, fancyhdr, hyperref, todonotes, caption}
\usepackage[margin = 1in]{geometry}

\newtheorem{theorem}{Theorem}[section]
\newtheorem{lemma}{Lemma}[section]
\newtheorem{proposition}{Proposition}[section]

\newtheorem{corollary}{Corollary}[section]

\newenvironment{proof}
      {\medskip\noindent{\bf Proof:}\hspace{1mm}}
      {\hfill$\Box$\medskip}

\input{epsf}

\makeatletter
\def\Ddots{\mathinner{\mkern1mu\raise\p@
\vbox{\kern7\p@\hbox{.}}\mkern2mu
\raise4\p@\hbox{.}\mkern2mu\raise7\p@\hbox{.}\mkern1mu}}
\makeatother
\def\d{\delta}
\def\e{\epsilon}

\newcommand{\abs}[1]{\left\lvert#1\right\rvert}
\newcommand{\CC}{\mathbb{C}}

\title{\vspace{-0.7cm}Graph removal lemmas}
\author{David Conlon\thanks{Mathematical Institute, Oxford OX1 3LB, United Kingdom. Email: {\tt david.conlon@maths.ox.ac.uk}. Supported by a Royal Society University Research Fellowship.} \and Jacob Fox\thanks{Department of Mathematics, MIT, Cambridge,
MA 02139-4307. Email: {\tt fox@math.mit.edu}. Supported by a Simons Fellowship and NSF Grant DMS-1069197.}}
\date{}
\begin{document}
\maketitle

\begin{abstract}
The graph removal lemma states that any graph on $n$ vertices with $o(n^{v(H)})$ copies of a fixed graph $H$ may be made $H$-free by removing $o(n^2)$ edges. Despite its innocent appearance, this lemma and its extensions have several important consequences in number theory, discrete geometry, graph theory and computer science. In this survey we discuss these lemmas, focusing in particular on recent improvements to their quantitative aspects.
\end{abstract}

\section{Introduction}

The triangle removal lemma states that for all $\e > 0$ there exists $\d > 0$ such that any graph on $n$ vertices with at most $\d n^3$ triangles may be made triangle-free by removing at most $\e n^2$ edges. This result, proved by Ruzsa and Szemer\'edi \cite{RuSz} in 1976, was originally stated in rather different language. 

The original formulation was in terms of the $(6,3)$-problem.\footnote{The two results are not exactly equivalent, though the triangle removal lemma may be proved by their method. A weak form of the triangle removal lemma, already sufficient for proving Roth's theorem, is equivalent to the Ruzsa-Szemer\'edi theorem. This weaker form states that any graph on $n$ vertices in which every edge is contained in exactly one triangle has $o(n^2)$ edges. This is also equivalent to another attractive formulation, known as the induced matching theorem. This states that any graph on $n$ vertices which is the union of at most $n$ induced matchings has $o(n^2)$ edges.} This asks for the maximum number of edges $f^{(3)}(n, 6, 3)$ in a $3$-uniform hypergraph on $n$ vertices such that no $6$ vertices contain $3$ edges. Answering a question of Brown, Erd\H{o}s and S\'os \cite{BES73}, Ruzsa and Szemer\'edi showed that $f^{(3)}(n, 6, 3) = o(n^2)$. Their proof used several iterations of an early version of Szemer\'edi's regularity lemma \cite{Sz76}. 

This result, developed by Szemer\'edi in his proof of the Erd\H{o}s-Tur\'an conjecture on arithmetic progressions in dense sets \cite{Sz1}, states that every graph may be partitioned into a small number of vertex sets so that the graph between almost every pair of vertex sets is random-like. Though this result now occupies a central position in graph theory, its importance only emerged over time. The resolution of the $(6,3)$-problem was one of the first indications of its strength. 

The Ruzsa-Szemer\'edi theorem was generalized by Erd\H{o}s, Frankl and R\"odl \cite{EFR86}, who showed that $f^{(r)}(n, 3r-3,3) = o(n^2)$, where $f^{(r)}(n, 3r-3,3)$ is the maximum number of edges in an $r$-uniform hypergraph such that no $3r-3$ vertices contain $3$ edges. One of the tools used by Erd\H{o}s, Frankl and R\"odl in their proof was a striking result stating that if a graph on $n$ vertices contains no copy of a graph $H$ then it may be made $K_r$-free, where $r = \chi(H)$ is the chromatic number of $H$, by removing $o(n^2)$ edges. The proof of this result used the modern formulation of Szemer\'edi's regularity lemma and is already very close, both in proof and statement, to the following generalization of the triangle removal lemma, known as the graph removal lemma.\footnote{The phrase `removal lemma' is a comparatively recent coinage. It seems to have come into vogue in about 2005 when the hypergraph removal lemma was first proved (see, for example, \cite{KNRSS05, NRS06, So05, T062}).} This was first stated explicitly in the literature by Alon, Duke, Lefmann, R\"odl and Yuster \cite{ADLRY94} and by F\"uredi \cite{Fu95} in 1994.\footnote{This was also the first time that the triangle removal lemma was stated explicitly, though the weaker version concerning graphs where every edge is contained in exactly one triangle had already appeared in the literature. The Ruzsa-Szemer\'edi theorem was usually \cite{FF87, FGR87, Fu92} phrased in the following suggestive form: if a $3$-uniform hypergraph is linear, that is, no two edges intersect on more than a single vertex, and triangle-free, then it has $o(n^2)$ edges. A more explicit formulation may be found in \cite{CEMcCSz91}.} Note that we use $v(H)$ to denote the number of vertices in a graph (or hypergraph) $H$.

\begin{theorem} \label{thm:graphremoval}
For any graph $H$ and any $\e > 0$, there exists $\d > 0$ such that any graph on $n$ vertices which contains at most $\d n^{v(H)}$ copies of $H$ may be made $H$-free by removing at most $\e n^2$ edges. 
\end{theorem}


It was already observed by Ruzsa and Szemer\'edi that the $(6,3)$-problem (and, thereby, the triangle removal lemma) is related to Roth's theorem on arithmetic progressions \cite{Ro}. This theorem states that for any $\delta > 0$ there exists an $n_0$ such that if $n \geq n_0$ then any subset of the set $[n] := \{1, 2, \dots, n\}$ of size at least $\delta n$ contains an arithmetic progression of length $3$. Letting $r_3(n)$ be the largest integer such that there exists a subset of the set $\{1, 2, \dots, n\}$ of size $r_3(n)$ containing no arithmetic progression of length $3$, this is equivalent to saying that $r_3(n) = o(n)$. Ruzsa and Szemer\'edi observed that $f^{(3)}(n,6,3) = \Omega(r_3(n) n)$. In particular, since $f^{(3)}(n,6,3) = o(n^2)$, this implies that $r_3(n) = o(n)$, yielding a proof of Roth's theorem.

It was further noted by Solymosi \cite{So} that the Ruzsa-Szemer\'edi theorem yields a stronger result of Ajtai and Szemer\'edi \cite{AjSz}. This result states that for any $\delta > 0$ there exists an $n_0$ such that if $n \geq n_0$ then any subset of the set $[n] \times [n]$ of size at least $\delta n^2$ contains a set of the form $\{(a, b), (a+d, b), (a, b+d)\}$. That is, dense subsets of the $2$-dimensional grid contain axis-parallel isosceles triangles. Roth's theorem is a simple corollary of this statement.

Roth's theorem is the first case of a famous result known as Szemer\'edi's theorem. This result, to which we alluded earlier, states that for any natural number $k \geq 3$ and any $\delta > 0$ there exists $n_0$ such that if $n \geq n_0$ then any subset of the set $[n]$ of size at least $\delta n$ contains an arithmetic progression of length $k$. This was first proved by Szemer\'edi \cite{Sz1} in the early seventies using combinatorial techniques and since then several further proofs have emerged. The most important of these are that by Furstenberg \cite{F77, FKO82} using ergodic theory and that by Gowers \cite{G98, G01}, who found a way to extend Roth's original Fourier analytic argument to general $k$. Both of these methods have been highly influential. 

Yet another proof technique was suggested by Frankl and R\"odl \cite{FR02}. They showed that Szemer\'edi's theorem would follow from the following generalization of Theorem~\ref{thm:graphremoval}, referred to as the hypergraph removal lemma. They proved this theorem for the specific case of $K_4^{(3)}$, the complete $3$-uniform hypergraph with $4$ vertices. This was then extended to all $3$-uniform hypergraphs in~\cite{NR03} and to $K_5^{(4)}$ in~\cite{RSk05}. Finally, it was proved for all hypergraphs by Gowers \cite{G06, G07} and, independently, by Nagle, R\"odl, Schacht and Skokan \cite{NRS06, RSk04}. Both proofs rely on extending Szemer\'edi's regularity lemma to hypergraphs in an appropriate fashion. 

\begin{theorem} \label{thm:hyperremoval}
For any $k$-uniform hypergraph $\mathcal{H}$ and any $\e > 0$, there exists $\d > 0$ such that any $k$-uniform hypergraph on $n$ vertices which contains at most $\d n^{v(\mathcal{H})}$ copies of $\mathcal{H}$ may be made $\mathcal{H}$-free by removing at most $\e n^k$ edges. 
\end{theorem}

As well as reproving Szemer\'edi's theorem, the hypergraph removal lemma allows one to reprove the multidimensional Szemer\'edi theorem. This theorem, originally proved by Furstenberg and Katznelson \cite{FK78}, states that for any natural number $r$, any finite subset $S$ of $\mathbb{Z}^r$ and any $\delta > 0$ there exists $n_0$ such that if $n \geq n_0$ then any subset of $[n]^r$ of size at least $\delta n^r$ contains a subset of the form $a\cdot S + d$, that is, a dilated and translated copy of $S$. That it follows from the hypergraph removal lemma was first observed by Solymosi \cite{So04}. This was the first non-ergodic proof of this theorem. A new proof of the special case $S=\{(0,0), (1,0), (0,1)\}$, corresponding to the Ajtai-Szemer\'edi theorem,  was given by Shkredov \cite{Shk06} using a Fourier analytic argument. Recently, a combinatorial proof of the density Hales-Jewett theorem, which is an extension of the multidimensional Szemer\'edi theorem, was discovered as part of the polymath project \cite{polymath}. 

As well as its implications in number theory, the removal lemma and its extensions are central to the area of computer science known as property testing. In this area, one would like to find fast algorithms to distinguish between objects which satisfy a certain property and objects which are far from satisfying that property. This field of study was initiated by Rubinfield and Sudan \cite{RuSu} and, subsequently, Goldreich, Goldwasser and Ron \cite{GGR} started the investigation of such property testers for combinatorial objects. Graph property testing has attracted a particular degree of interest. 

A classic example of property testing is to decide whether a given graph $G$ is $\e$-far from being triangle-free, that is, whether at least $\e n^2$ edges will have to removed in order to make it triangle-free. The triangle removal lemma tells us that if $G$ is $\e$-far from being triangle free then it must contain at least $\d n^3$ triangles for some $\d > 0$ depending only on $\e$. This furnishes a simple probabilistic algorithm for deciding whether $G$ is $\e$-far from being triangle-free. We choose $t = 2 \delta^{-1}$ triples of points from the vertices of $G$ uniformly at random. If $G$ is $\e$-far from being triangle-free then the probability that none of these randomly chosen triples is a triangle is $(1 - \delta)^t < e^{-t \delta} < \frac{1}{3}$. That is, if $G$ is $\e$-far from being triangle-free we will find a triangle with probability at least $\frac{2}{3}$, whereas if $G$ is triangle-free we will clearly find no triangles. The graph removal lemma may be used to derive a similar test for deciding whether $G$ is $\e$-far from being $H$-free for any fixed graph $H$.

In property testing, it is often of interest to decide not only whether a graph is far from being $H$-free but also whether it is far from being induced $H$-free. A subgraph $H'$ of a graph $G$ is said to be an induced copy of $H$ if there is a one-to-one map $f: V(H) \rightarrow V(H')$ such that $(f(u), f(v))$ is an edge of $H'$ if and only if $(u, v)$ is an edge of $H$. A graph $G$ is said to be induced $H$-free if it contains no induced copies of $H$ and $\e$-far from being induced $H$-free if we have to add and/or delete at least $\e n^2$ edges to make it induced $H$-free. Note that it is not enough to delete edges since, for example, if $H$ is the empty graph on two vertices and $G$ is the complete graph minus an edge, then $G$ contains only one induced copy of $H$, but one cannot simply delete edges from $G$ to make it induced $H$-free.

By proving an appropriate strengthening of the regularity lemma, Alon, Fischer, Krivelevich and Szegedy \cite{AFKS} showed how to modify the graph removal lemma to this setting. This result, which allows one to test for induced $H$-freeness, is known as the induced removal lemma.

\begin{theorem} \label{thm:inducedremoval}
For any graph $H$ and any $\e > 0$, there exists a $\d > 0$ such that any graph on $n$ vertices which contains at most $\d n^{v(H)}$ induced copies of $H$ may be made induced $H$-free by adding and/or deleting at most $\e n^2$ edges. 
\end{theorem}

A substantial generalization of this result, known as the infinite removal lemma, was proved by Alon and Shapira \cite{AlSh08} (see also \cite{LS10}). They showed that for each (possibly infinite) family  $\mathcal{H}$ of graphs and $\epsilon>0$ there is $\delta=\delta_{\mathcal{H}}(\epsilon)>0$ such that 
if a graph $G$ on $n$ vertices contains at most $\delta n^{v(H)}$ induced copies of $H$ for every graph $H$ in $\mathcal{H}$, then $G$ may be made induced $H$-free, for every $H \in \mathcal{H}$, by adding and/or deleting at most $\epsilon n^2$ edges. They then used this result to show that every hereditary graph property is testable, where a graph property is hereditary if it is closed under removal of vertices. These results were extended to $3$-uniform hypergraphs by Avart, R\"odl and Schacht \cite{ARS07} and to $k$-uniform hypergraphs by R\"odl and Schacht \cite{RS09}.

In this survey we will focus on recent developments, particularly with regard to the quantitative aspects of the removal lemma. In particular, we will discuss recent improvements on the bounds for the graph removal lemma, Theorem~\ref{thm:graphremoval}, and the induced graph removal lemma, Theorem~\ref{thm:inducedremoval}, each of which bypasses a natural impediment. 

The usual proof of the graph removal lemma makes use of the regularity lemma and gives bounds for the removal lemma which are of tower-type in $\e$. To be more specific, let $T(1) = 2$ and, for each $i \geq 1$, $T(i+1) = 2^{T(i)}$. The bounds that come out of applying the regularity lemma to removal then say that if $\delta^{-1} = T(\epsilon^{-c_H})$ then any graph with at most $\d n^{v(H)}$ copies of $H$ may be made $H$-free by removing at most $\e n^2$ edges. Moreover, this tower-type dependency is inherent in any proof employing regularity. This follows from an important result of Gowers \cite{G97} (see also \cite{CF12}) which states that the bounds that arise in the regularity lemma are necessarily of tower type. We will discuss this in more detail in Section~\ref{sec:usualremoval} below. 

Despite this obstacle, the following improvement was made by Fox \cite{F11}.

\begin{theorem}\label{thm:removalimproval}
For any graph $H$, there exists a constant $a_H$ such that if $\d^{-1} = T(a_H \log \e^{-1})$ then any graph on $n$ vertices which contains at most $\d n^{v(H)}$ copies of $H$ may be made $H$-free by removing at most $\e n^2$ edges. 
\end{theorem}

As is implicit in the bounds, the proof of this theorem does not make an explicit appeal to Szemer\'edi's regularity lemma. However, many of the ideas used are similar to ideas used in the proof of the regularity lemma. The chief difference lies in the fact that the conditions of the removal lemma (containing few copies of a given graph $H$) allow us to say more about the structure of these partitions. A simplified proof of this theorem will be the main topic of Section~\ref{sec:improvedremoval}.

Though still of tower-type, Theorem~\ref{thm:removalimproval} improves substantially on the previous bound. However, it remains very far from the best known lower bound on $\delta^{-1}$. The observation of Ruzsa and Szemer\'edi \cite{RuSz} that $f^{(3)}(n,6,3) = \Omega(r_3(n) n)$ allows one to transfer lower bounds for $r_3(n)$ to a corresponding lower bound for the triangle removal lemma. The best construction of a set containing no arithmetic progression of length $3$ is due to Behrend \cite{Be46} and gives a subset of $[n]$ with density $e^{-c\sqrt{\log n}}$. Transferring this to the graph setting yields a graph containing $\e^{c \log \e^{-1}} n^3$ triangles which cannot be made triangle-free by removing fewer than $\e n^2$ edges. This quasi-polynomial lower bound, $\delta^{-1} \geq \e^{-c \log \e^{-1}}$, remains the best known.\footnote{It is worth noting that the best known upper bound for Roth's theorem, due to Sanders \cite{San11}, is considerably better than the best upper bound for $r_3(n)$ that follows from triangle removal. This upper bound is $r_3(n) = O\left(\frac{(\log \log n)^5}{\log n} n\right)$. A recent result of Schoen and Shkredov \cite{SchShk}, building on further work of Sanders \cite{San12}, shows that any subset of $[n]$ of density $e^{-c (\frac{\log n}{\log \log n})^{1/6}}$ contains a solution to the equation $x_1 + \cdots + x_5 = 5 x_6$. Since arithmetic progressions correspond to solutions of $x_1 + x_2 = 2 x_3$, this suggests that the answer should be closer to the Behrend bound. The bounds for triangle removal are unlikely to impinge on these upper bounds for some time, if at all.}

The standard proof of the induced removal lemma uses the strong regularity lemma of Alon, Fischer, Krivelevich and Szegedy \cite{AFKS}. We will speak at length about this result in Section~\ref{sec:usualinduced}. Here it will suffice to say that, like the ordinary regularity lemma, the bounds which an application of this theorem gives for the induced removal lemma are necessarily very large. Let $W(1) = 2$ and, for $i \geq 1$, $W(i+1) = T(W(i))$. This is known as the wowzer function and its values dwarf those of the usual tower function.\footnote{To give some indication, we note that $W(2) = 4$, $W(3) = 65536$ and $W(4)$ is a tower of $2$s of height $65536$.} By using the strong regularity lemma, the standard proof shows that we may take $\delta^{-1} = W(a_H \epsilon^{-c})$ in the induced removal lemma, Theorem \ref{thm:inducedremoval}. Moreover, as with the ordinary removal lemma, such a bound is inherent in the application of the strong regularity lemma. This follows from recent results of Conlon and Fox \cite{CF12} and, independently, Kalyanasundaram and Shapira \cite{KSh12} showing that the bounds arising in strong regularity are necessarily of wowzer type.

In the other direction, Conlon and Fox \cite{CF12} showed how to bypass this obstacle and prove that the bounds for $\delta^{-1}$ are at worst a tower in a power of $\epsilon^{-1}$.

\begin{theorem} \label{thm:inducedimproval}
There exists a constant $c>0$ such that, for any graph $H$, there exists a constant $a_H$ such that if $\d^{-1} = T(a_H \e^{-c})$ then any graph on $n$ vertices which contains at most $\d n^{v(H)}$ induced copies of $H$ may be made induced $H$-free by adding and/or deleting at most $\e n^2$ edges.
\end{theorem}

A discussion of this theorem will form the subject of Section~\ref{sec:improvedinduced}. The key observation here is that the strong regularity lemma is used to prove an intermediate statement (Lemma~\ref{strongeasycor} below) which then implies the induced removal lemma. This intermediate statement may be proved without recourse to the full strength of the strong regularity lemma. There are also some strong parallels with the proof of Theorem~\ref{thm:removalimproval} which we will draw attention to in due course. 

In Section~\ref{sec:infinite}, we present the proof of Alon and Shapira's infinite removal lemma. In another paper, Alon and Shapira \cite{AlSh08a} showed that the dependence in the infinite removal lemma can depend heavily on the family $\mathcal{H}$. They proved that for every function $\delta:(0,1) \rightarrow (0,1)$,  there exists a family $\mathcal{H}$ of graphs such that any $\delta_{\mathcal{H}}:(0,1) \rightarrow (0,1)$ which satisfies the infinite removal lemma for $\mathcal{H}$ satisfies $\delta_{\mathcal{H}}=o(\delta)$. However, such examples are rather unusual and the proof presented in Section~\ref{sec:infinite} of the infinite removal lemma implies that for many commonly studied families $\mathcal{H}$ of graphs the bound on $\delta_{\mathcal{H}}^{-1}$ is only tower-type, improving the wowzer-type bound from the original proof. 

Our discussions of the graph removal lemma and the induced removal lemma will occupy the bulk of this survey but we will also talk about some further recent developments in the study of removal lemmas. These include arithmetic removal lemmas (Section~\ref{sec:arithmeticremoval}) and the recently developed sparse removal lemmas which hold for subgraphs of sparse random and pseudorandom graphs (Section~\ref{sec:sparseremoval}). We will conclude with some further comments on related topics.

\section{The graph removal lemma} \label{sec:removal}

In this section we will discuss the two proofs of the removal lemma, Theorem~\ref{thm:graphremoval}, at length. In Section~\ref{sec:usualremoval}, we will talk about the regularity lemma and the usual proof
 of the removal proof. Then, in Section~\ref{sec:improvedremoval}, we will consider a simplified variant of the second author's recent proof \cite{F11}, showing how it connects to the weak regularity lemma of Frieze and Kannan \cite{FrKa, FrKa1}.

\subsection{The standard proof} \label{sec:usualremoval}

We begin with the proof of the regularity lemma and then deduce the removal lemma. For vertex subsets $S,T$ of a graph $G$, we let $e_G(S,T)$ denote the number of pairs in $S \times T$ that are edges of $G$ and $d_G(S,T)=\frac{e_G(S,T)}{|S||T|}$ denote the fraction of pairs in $S \times T$ that are edges of $G$. For simplicity of notation, we drop the subscript if the graph $G$ is clear from context. A pair $(S,T)$  of subsets is {\it $\epsilon$-regular} if, for all subsets $S' \subset S$ and $T' \subset T$ with $|S'| \geq \epsilon|S|$ and $|T'|\geq \epsilon |T|$, we have $|d(S',T')-d(S,T)| \leq \epsilon$. Informally, a pair of subsets is $\epsilon$-regular with a small $\epsilon$ if the edges between $S$ and $T$ are uniformly distributed among large subsets.  

Let $G=(V,E)$ be a graph and $P:V=V_1 \cup \ldots \cup V_k$ be a vertex partition of $G$. The partition of $P$ is {\it equitable} if each pair of parts differ in size by at most $1$. The partition $P$ is {\it $\epsilon$-regular} if all but at most $\epsilon k^2$ pairs of parts $(V_i,V_j)$ are $\epsilon$-regular. We next state Szemer\'edi's regularity lemma. 

\begin{lemma}\label{szemreg}
For every $\epsilon>0$, there is $K=K(\epsilon)$ such that every graph $G=(V,E)$ has an equitable, $\epsilon$-regular vertex partition into at most $K$ parts. Moreover, we may take $K$ to be a tower of height  $O(\epsilon^{-5})$. 
\end{lemma}

Let $q:[0,1] \rightarrow \mathbb{R}$ be a convex function. For vertex subsets $S,T \subset V$ of a graph $G$, let $q(S,T)=q(d(S,T))\frac{|S||T|}{|V|^2}$. For partitions $\mathcal{S}:S =S_1 \cup \ldots \cup S_a$ and $\mathcal{T}:T=T_1 \cup \ldots \cup T_b$, let $q(\mathcal{S},\mathcal{T})=\sum_{1 \leq i \leq a, 1 \leq j \leq b}q(S_i,T_j)$. For a vertex partition $P:V=V_1 \cup \ldots \cup V_k$ of $G$, define the mean-$q$ density to be $$q(P)=q(P,P)=\sum_{1 \leq i,j \leq k} q(V_i,V_j).$$

We next state some simple properties which follow from Jensen's inequality using the convexity of $q$. A {\it refinement} of a partition $P$ of a vertex set $V$ is another partition $Q$ of $V$ such that every part of $Q$ is a subset of a part of $P$. 

\begin{proposition} \label{firstprop}
\begin{enumerate} 
\item For partitions $\mathcal{S}$ and $\mathcal{T}$ of vertex subsets $S$ and $T$, we have $q(\mathcal{S},\mathcal{T}) \geq q(S,T)$. 
\item If $Q$ is a refinement of $P$, then $q(Q) \geq q(P)$. 
\item If $d=d(G)=d(V,V)$ is the edge density of $G$, then, for any vertex partition $P$, $$q(d) \leq q(P) \leq dq(1)+(1-d)q(0).$$
\end{enumerate}
\end{proposition} 

The first and second part of Proposition \ref{firstprop} show that by refining a vertex partition the mean-$q$ density cannot decrease, while the last part gives the range of possible values for $q(P)$ if we only know the edge density $d$ of $G$.

The convex function $q(x)=x^2$ for $x \in [0,1]$ is chosen in the standard proof of the graph regularity lemma and we will do the same for the rest of this subsection. The following lemma is the key claim for the proof of the regularity lemma. The set-up is that we have a partition $P$ which is not $\epsilon$-regular. For each pair $(V_i,V_j)$ of parts of $P$ which is not $\epsilon$-regular, there are a pair of witness subsets $V_{ij},V_{ji}$ to the fact that the pair of parts is not $\epsilon$-regular. We consider the coarsest refinement $Q$ of $P$ so that each witness subset is the union of parts of $Q$. The lemma concludes that the number of parts of $Q$ is at most exponential in the number of parts of $P$ and, using a Cauchy-Schwarz defect inequality, that the mean-$q$ density of the partition $Q$ is substantially larger than the mean-$q$ density of $P$. Because it simplifies our calculations a little, we will assume, when we say a partition is equitable, that it is exactly equitable, that is, that all parts have precisely the same size. This does not affect our results substantially but simplifies the presentation.

\begin{lemma} \label{keyregclaim}
If an equitable partition $P:V=V_1 \cup \ldots \cup V_k$ is not $\epsilon$-regular then there is a refinement $Q$ of $P$ into at most $k2^k$ parts for which $q(Q) \geq q(P)+\epsilon^5$. 
\end{lemma} 
\begin{proof}
For each pair $(V_i,V_j)$ which is not $\epsilon$-regular, there are subsets $V_{ij} \subset V_i$ and $V_{ji} \subset V_j$ with $|V_{ij}| \geq \epsilon|V_i|$ and $|V_{ji}| \geq \epsilon |V_j|$ such that $|d(V_{ij},V_{ji})-d(V_i,V_j)| \geq \epsilon$. For each part $V_j$ such that $(V_i,V_j)$ is not $\epsilon$-regular, we have a partiton $P_{ij}$ of $V_i$ into two parts $V_{ij}$ and $V_i \setminus V_{ij}$. Let $P_i$ be the partition of $V_i$ which is the common refinement of these at most $k-1$ partitions of $V_i$, so $P_i$ has at most $2^{k-1}$ parts. We let $Q$ be the partition of $V$ which is the union of the $k$ partitions of the form $P_i$, so $Q$ has at most $k2^{k-1}$ parts. 
We have \begin{eqnarray*} q(Q)-q(P) & = & \sum_{i,j}\left(q(P_i,P_j)-q(V_i,V_j)\right)   
\\ & \geq &  \sum_{(V_i,V_j)~\textrm{irregular}}\left(q(P_i,P_j)-q(V_i,V_j)\right)  
\\ & \geq &  \sum_{(V_i,V_j)~\textrm{irregular}}\left(q(P_{ij},P_{ji})-q(V_i,V_j)\right)  
 \\ & = &  \sum_{(V_i,V_j)~\textrm{irregular}}\sum_{U \in P_{ij},W \in P_{ji}}\frac{|U||W|}{|V|^2}\left(d(U,W)-d(V_i,V_j)\right)^2
\\ & \geq  &  \sum_{(V_i,V_j)~\textrm{irregular}}\frac{|V_{ij}||V_{ji}|}{|V|^2}\left(d(V_{ij},V_{ji})-d(V_i,V_j)\right)^2 
\\ & \geq & \epsilon k^2\left(\frac{\epsilon}{k}\right)^2\epsilon^2 \\ & = & \epsilon^5,
\end{eqnarray*}
where the first and third inequalities are by noting that the summands are nonnegative and the second inequality follows from the first part of Proposition \ref{firstprop}, which shows that the mean-$q$ density cannot decrease when taking a refinement. In the fourth inequality, we used that $|V_{ij}| \geq \e |V_i| \geq \frac{\e}{k} |V|$ and similarly for $|V_{ji}|$.  Finally, the equality in the fourth line follows from the identity 
\[\sum_{U \in P_{ij}, W \in P_{ji}} |U||W| d(V_i, V_j) = \sum_{U \in P_{ij}, W \in P_{ji}} |U||W| d(U, W),\]
which counts $e(V_i,V_j)$ in two different ways. This completes the proof.
\end{proof} 

The next lemma, which is rather standard, shows that for any vertex partition
$Q$, there is a vertex equipartition $P'$ with a similar number of parts to $Q$
and mean-square density not much smaller than the mean-square density of $Q$.
It is useful in density increment arguments where at each stage one would like
to work with an equipartition. It is proved by first arbitrarily partitioning each part of $Q$ into parts of order $|V|/t$, except possibly one additional remaining smaller part, and then arbitrarily partitioning the 
union of the smaller remaining parts into parts of order $|V|/t$. 

\begin{lemma}\label{makeequip} Let $G=(V,E)$ be a graph and $Q:V=V_1 \cup
\ldots \cup V_\ell$ be a vertex partition into $k$ parts.  Then, for $q(x) = x^2$, there is an equitable
partition $P'$ of $V$ into $t$ parts such that $q(P') \geq q(Q)-2\frac{\ell}{t}$.
\end{lemma}

Combining Lemmas \ref{keyregclaim} and \ref{makeequip} with $t=4\epsilon^{-5}|Q| \leq \epsilon^{-5}k2^{k+2}$, we obtain the following corollary. 

\begin{corollary} \label{keyregclaim2}
If an equitable partition $P:V=V_1 \cup \ldots \cup V_k$ is not $\epsilon$-regular then there is an equitable refinement $P'$ of $P$ into at most $\epsilon^{-5}k2^{k+2}$ parts for which $q(P') \geq q(P)+\epsilon^5/2$. 
\end{corollary} 

We next show how Szemer\'edi's regularity lemma, Lemma \ref{szemreg}, can be quickly deduced from this result. 

\begin{proof}To prove the regularity lemma, we start with the trivial partition $P_0$ into one part, and iterate the above corollary to obtain a sequence $P_0,P_1,\ldots,P_s$ of equitable partitions with $q(P_{i+1}) \geq q(P_{i})+\epsilon^5/2$ until we arrive at an equitable $\epsilon$-regular partition $P_s$. As the mean-square density of each partition has to lie between $0$ and $1$, after at most $2\epsilon^{-5}$ iterations we arrive at the equitable $\epsilon$-regular partition $P_s$ with $s \leq 2\epsilon^{-5}$. The number of parts increases by one exponential in each iteration, giving the desired number of parts in the regularity partition. This completes the proof of Szemer\'edi's regularity lemma.
\end{proof}

The constructions of Gowers \cite{G97} and the authors \cite{CF12} show that the tower-type bound on the number of parts in Szemer\'edi's regularity lemma is indeed necessary. In particular, the construction in \cite{CF12} shows that $K(\epsilon)$ in Lemma \ref{szemreg} is at least a tower of twos of height $\Omega(\epsilon^{-1})$. The constructions are formed by reverse engineering the upper bound proof. We construct a sequence $P_0,\ldots,P_s$ of partitions with $s=\Omega(\epsilon^{-1})$. As in the upper bound proof, each partition in the sequence uses exponentially more parts than the previous partition in the sequence. We may choose the edges using these partitions and some randomness so as to guarantee that none of these partitions (except the last) are $\epsilon$-regular. Furthermore, we can guarantee that any partition that is $\epsilon$-regular must be close to being a refinement of the last partition in this sequence. This implies that the number of parts must be at least roughly $|P_s|$. 

We next prove the graph removal lemma, Theorem \ref{thm:graphremoval}, from the regularity lemma. 

\begin{proof}
Let $m$ denote the number of edges of $H$, so $m \leq {h \choose 2}$. Let $\gamma=\frac{\epsilon^h}{4h}$ and $\delta=(2h)^{-2h}\epsilon^{m}K^{-h}$, where $K=K(\gamma)$ is as in the regularity lemma. We apply the regularity lemma to $G$ and obtain an equitable, $\gamma$-regular partition into $k \leq  K$ parts. If the number $n$ of vertices of $G$ satisfies $n<\delta^{-1/h}$, then the number of copies of $H$ in $G$ is at most $\delta n^h < 1$ and $G$ is $H$-free, in which case there is nothing to prove. So we may assume $n \geq \delta^{-1/h}$. We obtain a subgraph $G'$ of $G$ by removing  edges of $G$ between all pairs of parts which are not $\gamma$-regular or which have edge density at most $\epsilon$. As there are at most $\gamma k^2$ ordered pairs of parts which are not $\gamma$-regular and each part has order at most $2n/k$, at most $(\gamma k^2/2)(2n/k)^2=2\gamma n^2$ edges are deleted between pairs of parts which are not $\gamma$-regular.  The number of edges between parts which have edge density at most $\epsilon$ is at most $\epsilon n^2/2$. Hence, the number of edges of $G$ deleted to obtain $G'$ is at most $2\gamma n^2+\epsilon n^2 /2 < \epsilon n^2$. If $G'$ is $H$-free, then we are done. 

Assume for contradiction that $G'$ is not $H$-free. A copy of $H$ in $G'$ must have its edges going between pairs of parts which are both $\gamma$-regular and have density at least $\epsilon$. Hence, there is a mapping from $V(H)$ to the partition of $V(G)$ so that each edge of $H$ maps to a pair of parts which is both $\gamma$-regular and have edge density at least $\epsilon$. But the following standard counting lemma (see, e.g., Lemma 3.2 in Alon, Fischer, Krivelevich and Szegedy \cite{AFKS} for a minor variant) shows that the number of labeled copies of $H$ in $G'$ (and hence in $G$) is at least $2^{-h}\epsilon^m(n/2k)^h > h!\delta n^h$. This contradicts that $G$ has at most $\delta n^h$ copies of $H$, completing the proof. 
\end{proof}

\begin{lemma} \label{lem:AFKSusualcount}

If $H$ is a graph with vertices $1,\ldots,h$ and $m$ edges and $G$ is a graph with not necessarily disjoint vertex subsets $W_1,\ldots,W_h$ such that $|W_i| \geq \gamma^{-1}$ for $1 \leq i \leq h$ and, for every edge $(i,j)$ of $H$, the pair $(W_i,W_j)$ is $\gamma$-regular with density $d(W_i,W_j)>\epsilon$ and  $\gamma \leq \frac{\epsilon^h}{4h}$, then $G$ contains at least
$2^{-h}\epsilon^{m}|W_1| \times \cdots \times |W_h|$
labeled copies of $H$ with the copy of vertex $i$ in $W_i$.
\end{lemma}

The standard proof of this counting lemma uses a greedy embedding strategy. One considers embedding the vertices one at a time, using the regularity condition to maintain the property that at each step where vertex $i$ of $H$ is not yet embedded the set of vertices of $G$ which could potentially be used to embed vertex $i$ is large. 

\subsection{An improved bound} \label{sec:improvedremoval}

A partition $P:V=V_1 \cup \ldots \cup V_k$ of the vertex set of a graph $G=(V,E)$ is 
{\it weak $\epsilon$-regular} if, for all subsets $S,T \subset V$, we have $$\left|e(S,T)-\sum_{1 \leq i,j \leq k} |S \cap V_i||T \cap V_j|d(V_i,V_j)\right| \leq \epsilon |V|^2.$$ 
That is, the density between two sets may be approximated by taking a weighted average over the densities between the sets which they intersect.

The Frieze-Kannan weak regularity lemma \cite{FrKa, FrKa1} states that any graph has such a weak regular partition. 

\begin{lemma} \label{lem:FK} 
Let $R(\epsilon)=2^{c\epsilon^{-2}}$, where $c$ is an absolute constant. For every graph $G=(V,E)$ and every equitable partition $P$ of $G$ into $k$ parts, there is an equitable partition $P'$ which is a refinement of $P$ into at most $kR(\epsilon)$ parts which is weak $\epsilon$-regular. 
\end{lemma}

Unlike the usual regularity lemma, the bounds in Lemma \ref{lem:FK} are quite reasonable.\footnote{They are also sharp, that is, there are graphs for which the minimum number of parts in any weak $\e$-regular partition is $2^{\Omega(\e^{-2})}$. This was proved in \cite{CF12} (see also \cite{AFKK}).} It is therefore natural to try to apply it to prove the removal lemma. However, it seems unlikely that this lemma is itself sufficient to prove the removal lemma, since it only gives control over edge densities of a global nature. 

However, as noted by Tao \cite{T06} (see also \cite{RS10}), one can prove a stronger theorem by simply iterating the Frieze-Kannan weak regularity lemma.\footnote{We will say more about this sort of iteration in Section~\ref{sec:usualinduced} below.} Tao developed this lemma to give an alternative proof of the regularity lemma\footnote{More recently, Conlon and Fox \cite{CF12} showed that it is also closely related to the regular approximation lemma. This lemma, which arose in the study of graph limits by Lov\'asz and Szegedy \cite{LS07} and also in work on the hypergraph generalization of the regularity lemma by R\"odl and Schacht \cite{RS07}, says that by adding and/or deleting a small number of edges in a graph $G$, we may find another graph $G'$ which admits very fine regular partitions. We refer the reader to \cite{CF12} and \cite{RS10} for further details.} which extended more easily to hypergraphs \cite{T062}. Here we use it to improve the bounds for removal.

\begin{lemma} \label{generalreg}
Let $q:[0,1] \rightarrow \mathbb{R}$ be a convex function, $G$ be a graph with $d=d(G)$, $f:\mathbb{N} \rightarrow [0,1]$ be a decreasing function and $r=\left(dq(1)+(1-d)q(0)-q(d)\right)/ \gamma$. Then there are equitable partitions $P$ and $Q$ with $Q$ a refinement of $P$ satisfying $q(Q) \leq q(P)+\gamma$,  $Q$ is weak $f(|P|)$-regular and $|Q| \leq t_r$, where $t_0=1$, $t_{i}=t_{i-1}R(f(t_{i-1}))$ for $1 \leq i \leq r$ and $R(x)=2^{cx^{-2}}$ as in the Frieze-Kannan weak regularity lemma. 
\end{lemma}

The proof of Lemma \ref{generalreg} is quite similar to the proof of Szemer\'edi's regularity lemma discussed in the previous subsection. One starts with the trivial partition $P_0$ of $V$ into one part. We then apply Lemma \ref{lem:FK} repeatedly to construct a sequence of partitions $P_0,P_1,\ldots$ so that $P_{i+1}$ is weak $f(|P_i|)$-regular. If $q(P_{i+1}) > q(P_i)+\gamma$, then we continue with this process. Otherwise, $q(P_{i+1}) \leq q(P_i)+\gamma$, so we set $Q=P_{i+1}$ and $P=P_i$ and stop the process. This process must stop within $r$ iterations as the third part of Proposition \ref{firstprop} shows that the mean-$q$ density lies in an interval of length $r\gamma$. 

Rather than using the usual $q(x) = x^2$, we will use the convex function $q$ on $[0,1]$ defined by $q(0)=0$ and $q(x)=x \log x$ for $x \in (0,1]$. This entropy function is central to the proof since it captures the extra structural information coming from Lemma~\ref{key} below in a concise fashion. Note that the last part of Proposition \ref{firstprop} implies that $d\log d \leq q(P) \leq 0$ for every partition $P$.

The next lemma is a counting lemma that complements the Frieze-Kannan weak regularity lemma. As one might expect, this lemma gives a global count for the number of copies of $H$, whereas the counting lemma associated with the usual regularity lemma gives a means of counting copies of $H$ between any $v(H)$ parts of the partition which are pairwise regular. Its proof, which we omit, is by a simple telescoping sum argument.

\begin{lemma}\label{weakcount} (\cite{BCLSV}, Theorem 2.7 on page 1809) Let $H$ be a graph on $\{1,\ldots,h\}$ with $m$ edges. 
Let $G=(V,E)$ be a graph on $n$ vertices and $Q:V=V_1 \cup \ldots \cup V_t$ be a vertex partition which is weak $\epsilon$-regular. The number of homomorphisms from $H$ to $G$ is within $\epsilon m n^h$ of $$\sum_{1 \leq i_1,\ldots,i_h \leq t} \, \prod_{(r,s) \in E(H)}d(V_{i_r},V_{i_s})\prod_{a=1}^h |V_{i_a}|.$$
\end{lemma}

Let $P$ and $Q$ be vertex partitions of a graph $G$ with $Q$ a refinement of $P$. A pair $(V_i,V_j)$ of parts of $P$ is {\it $(\alpha,c)$-shattered} by $Q$ if at least a $c$-fraction of the pairs $(u,v) \in V_i \times V_j$ go between pairs of parts of $Q$ with edge density between them less than $\alpha$. 

One of the key components of the proof is the following lemma, which says that if $P$ and $Q$ are vertex partitions like those given by Lemma~\ref{generalreg}, then there are many pairs of vertex sets in $P$ which are shattered by $Q$.

\begin{lemma}\label{key}
Let $H$ be a graph on $\{1,\ldots,h\}$ with $m$ edges and let $\alpha>0$. Suppose $G$ is a graph on $n$ vertices for which there are less than $\delta n^h$ homomorphisms of $H$ into $G$, where $\delta=\frac{1}{4}\alpha^m(2k)^{-h}$. Suppose $P$ and $Q$ are equitable vertex partitions of $G$ with $|P|=k \leq n$ and $Q$  is a refinement of $P$ which is weak $f(k)$-regular, where $f(k)=\frac{1}{4m}\alpha^m(2k)^{-h}$. For every $h$-tuple $V_1,\ldots,V_h$ of parts of $P$, there is an edge $(i,j)$ of $H$ for which the pair $(V_i,V_j)$ is $(\alpha,\frac{1}{2m})$-shattered by $Q$. 
\end{lemma}
\begin{proof}
As $|P|=k \leq n$, we have $|V_i| \geq \frac{n}{2k}$ for each $i$. Let $Q_i$ denote the partition of $V_i$ which consists of the parts of $Q$ which are subsets of $V_i$. Consider an $h$-tuple $(v_1,\ldots,v_h) \in V_1 \times \cdots \times V_h$ picked uniformly at random. Also consider the event $E$ that, for each edge $(i,j)$ of $H$, the pair $(v_i,v_j)$ goes between parts of $Q_i$ and $Q_j$ with density at least $\alpha$. If $E$ occurs with probability at least $1/2$, as $Q$ is weak $f(k)$-regular, Lemma \ref{weakcount} implies that the number of homomorphisms of $H$ into $G$ where the copy of vertex $i$ is in $V_i$ for $1 \leq i \leq h$ is at least 
$$\frac{1}{2}\alpha^{m}\prod_{i=1}^h |V_i| - m f(k) n^h \geq \left(\frac{1}{2}\alpha^m(2k)^{-h}-m f(k) \right) n^h = \delta n^h,$$ contradicting that there are less than $\delta n^h$ homomorphisms of $H$ into $G$. So $E$ occurs with probability less than $1/2$. Hence, for at least $1/2$ of the $h$-tuples $(v_1,\ldots,v_h) \in V_1 \times \cdots \times V_h$, there is an edge $(i,j)$ of $H$ such that the pair $(v_i,v_j)$ goes between parts of $Q_i$ and $Q_j$ with density less than $\alpha$. This implies that for at least one edge $(i,j)$ of $H$, the pair $(V_i,V_j)$ is $(\alpha,\frac{1}{2m})$-shattered by $Q$. 
 \end{proof}

We will need the following lemma from \cite{F11} which tells us that if a pair of parts from $P$ is shattered by $Q$ then there is an increment in the mean-entropy density. Its proof is by a simple application of Jensen's inequality. 

\begin{lemma} \label{defectinequality}(\cite{F11}, Lemma 7 on page 570) Let $q:[0,1] \rightarrow \mathbb{R}$ be the convex function given by $q(0)=0$ and $q(x)=x\log x$ for $x>0$.  Let $\epsilon_1,\ldots,\epsilon_r$ and $d_1,\ldots,d_r$ be nonnegative real numbers with $\sum_{i=1}^r \epsilon_i=1$ and $d=\sum_{i=1}^s \epsilon_id_i$. Suppose $\beta<1$ and $I \subset [r]$ is such that $d_i \leq \beta d$ for $i \in I$ and let $s=\sum_{i \in I} \epsilon_i$. Then $$\sum_{i=1}^r \epsilon_i q(d_i) \geq q(d)+(1-\beta+q(\beta))sd.$$
\end{lemma}

We are now ready to prove Theorem~\ref{thm:removalimproval} in the following precise form. 

\begin{theorem}
Let $H$ be a graph on $\{1,\ldots,h\}$ with $m$ edges. Let $\epsilon>0$ and $\delta^{-1}$ be a tower of twos of height $8 h^4 \log \epsilon^{-1}$. If $G$ is a graph on $n$ vertices in which at least $\epsilon n^2$ edges need to be removed to make it $H$-free, then $G$ contains at least $\delta n^h$ copies of $H$. 
\end{theorem}
\begin{proof} 
Suppose for contradiction that there is a graph $G$ on $n$ vertices in which at least $\epsilon n^2$ edges need to be removed from $G$ to delete all copies of $H$, but $G$ contains fewer than $\delta n^h$ copies of $H$. If $n \leq \delta^{-1/h}$, then the number of copies of $H$ in $G$ is less than $\delta n^h \leq 1$, so $G$ is $H$-free, contradicting that at least $\epsilon n^2$ edges need to be removed to make the graph $H$-free. Hence, $n > \delta^{-1/h}$. Note that the number of mappings from $V(H)$ to $V(G)$ which are not one-to-one is $n^h-h!{n \choose h} \leq h^2n^{h-1} < h^2\delta^{1/h}n^h$. Let $\delta'=2h^2\delta^{1/h}$, so the number of homomorphisms from $H$ to $G$ is at most $\delta' n^h$. 

The graph $G$ contains at least $\epsilon n^2/m$ edge-disjoint copies of $H$. Let $G'$ be the graph on the same vertex set which consists entirely of the at least $\epsilon n^2/m$ edge-disjoint copies of $H$. Then $d(G') \geq m \cdot \epsilon /m = \epsilon$ and $G'$ consists of $\frac{d(G')}{m} n^2$ edge-disjoint copies of $H$. We will show that there are at least $\delta' n^h$ homomorphisms from $H$ to $G'$ (and hence to $G$ as well). For the rest of the argument, we will assume the underlying graph is $G'$. 

Let $\alpha=\frac{\epsilon}{8m}$. Apply Lemma \ref{generalreg} to $G'$ with $f(k)=\frac{1}{4m}\alpha^m(2k)^{-h}$ and $\gamma=\frac{d(G')}{2h^4}$. Note that $r$ as in Lemma \ref{generalreg} is $$r=d(G')\log (1/d(G'))/\gamma =  2h^4\log (1/d(G')) \leq 2h^4\log \epsilon^{-1}.$$ Hence, we get a pair of equitable vertex partitions $P$ and $Q$, with $Q$ a refinement of $P$, $q(Q) \leq q(P)+\gamma$, $Q$ is weak $f(|P|)$-regular and $|Q|$ is at most a tower of twos of height $3r \leq 6h^4\log \epsilon^{-1}$.  Let $V_1,\ldots,V_k$ denote the parts of $P$ and $Q_i$ denote the partition of $V_i$ consisting of the parts of $Q$ which are subsets of $V_i$. 

Suppose that $(V_a,V_b)$ is a pair of parts of $P$ with edge density $d=d(V_a,V_b) \geq \epsilon/m$ which is $(\alpha,\frac{1}{2m})$-shattered by $Q$. Note that $\alpha \leq d/8$.  Arbitrarily order the pairs $U_i \times W_i \in Q_a \times Q_b$, letting $d_i=d(U_i,W_i)$ and $\epsilon_i=\frac{|U_i||W_i|}{|V_a||V_b|}$, so that the conditions of Lemma \ref{defectinequality} with $\beta=1/8$ are satisfied. Applying Lemma  \ref{defectinequality}, we get, since $q(\beta) = -\frac{1}{8} \log 8 = - \frac{3}{8}$, that $$q(Q_a,Q_b)-q(V_a,V_b) \geq  (1-\beta+q(\beta))\frac{1}{2m} d(V_a,V_b) |V_a||V_b|/n^2 \geq \frac{1}{4m}e(V_a,V_b)/n^2.$$

Note that $$q(Q)-q(P)=\sum_{1 \leq a,b \leq k} (q(Q_a,Q_b)-q(V_a,V_b)),$$ which shows that $q(Q)-q(P)$ is the sum of nonnegative summands. 

There are at most $\frac{\epsilon}{m}n^2/2$ edges of $G'$ going between pairs of parts of $P$ with density at most $\frac{\epsilon}{m}$. Hence, at least $1/2$ of the edge-disjoint copies of $H$ making up $G'$ have all its edges going between pairs of parts of $P$ of density at least $\frac{\epsilon}{m}$. By Lemma \ref{key}, for each copy of $H$, at least one of its edges goes between a pair of parts of $P$ which is $(\alpha,\frac{1}{2m})$-shattered by $Q$. Thus, $$q(Q)-q(P) \geq \sum  \frac{1}{4m} e (V_a,V_b)/n^2 \geq \frac{1}{4m} \cdot \frac{d(G')}{2m} = \frac{d(G')}{8m^2} > \gamma,$$
where the sum is over all ordered pairs $(V_a,V_b)$ of parts of $P$ which are $(\alpha,\frac{1}{2m})$-shattered by $Q$ and with $d(V_a,V_b) \geq \frac{\epsilon}{m}$. This contradicts $q(Q) \leq q(P)+\gamma$ and completes the proof. 
\end{proof}

\section{The induced removal lemma}

As in the last section, we will again discuss two different proofs of the induced removal lemma, Theorem~\ref{thm:inducedremoval}. In Section~\ref{sec:usualinduced}, we will discuss the proof of Alon, Fischer, Krivelevich and Szegedy \cite{AFKS}, which uses their strong regularity lemma and gives a wowzer-type bound. In Section~\ref{sec:improvedinduced}, we will examine the authors' recent proof~\cite{CF12} of a tower-type bound. We will discuss Alon and Shapira's generalization of the induced removal lemma, which applies to infinite families of graphs, in Section~\ref{sec:infinite}.

\subsection{The usual proof} \label{sec:usualinduced}

For an equitable partition $P = \{V_i|1 \leq i \leq k\}$
of $V(G)$ and an equitable refinement $Q = \{V_{i,j}|1 \leq i \leq
k,1 \leq j \leq \ell\}$ of $P$, we say that $Q$ is {\it
$\epsilon$-close} to $P$ if the following is satisfied. All $1\leq i
\leq i' \leq k$ but at most $\epsilon k^2$ of them are such that, for all $1\leq
j,j' \leq \ell$ but at most $\epsilon \ell^2$ of them, 
$|d(V_i,V_{i'})-d(V_{i,j},V_{i',j'})|<\epsilon$ holds. This notion roughly says
that $Q$ is an approximation of $P$. The strong regularity lemma of  Alon, Fischer, Krivelevich and Szegedy \cite{AFKS} is now as follows.

\begin{lemma}\label{strongreg} {\bf (Strong regularity lemma)} For every
function $f:\mathbb{N} \rightarrow (0,1)$ there exists a number $S = S(f)$ with
the following property. For every graph $G=(V,E)$, there is an equitable
partition $P$ of the vertex set $V$ and an equitable refinement
$Q$ of $P$ with $|Q| \leq S$ such that the
partition $P$ is $f(1)$-regular, the partition $Q$ is
$f(|P|)$-regular and $Q$ is $f(1)$-close to $P$.
\end{lemma}

That is, there is a regular partition $P$ and a refinement $Q$ such that $Q$ is very regular and yet the densities between parts $(V_{i, j}, V_{i',j'})$ of $Q$ are usually close to the densities between the parts $(V_i, V_{i'})$ of $P$ containing them. 

Let $f(1) = \e$. Here, and throughout this section, we let $q(x)=x^2$ be the square function as in the proof of Szemer\'edi's regularity lemma in the previous section. The condition that $Q$ be $\epsilon$-close to $P$ is equivalent,  up to a polynomial change in
$\epsilon$, to $q(Q) \leq q(P)+\epsilon$. Indeed, if $Q$ is $\epsilon$-close to $P$, then $q(Q) \leq q(P) + O(\epsilon)$, while if $q(Q) \leq q(P)+\epsilon$, then $Q$ is $O(\epsilon^{1/4})$-close to $P$. A version of this statement is present in Lemma 3.7 of \cite{AFKS}. As it is sufficient and more convenient to work with mean-square density instead of $\epsilon$-closeness, we do so from now on. That is, we replace the third condition in the regularity lemma with the condition that $q(Q) \leq q(P) + \e$.

With this observation, the proof of the strong removal lemma becomes quite straightforward. Note that we may assume that $f$ is a decreasing function by replacing it, if necessary, with the function given by $f'(i) = \min_{1 \leq j \leq i} f(j)$. We consider a series of partitions $P_1, P_2, \dots$, where $P_1$ is an $f(1)$-regular partition and $P_{i+1}$ is an $f(|P_i|)$-regular refinement of the partition $P_i$. Since $P_{i+1}$ is a refinement of $P_i$ we know that the mean-square density must have increased, that is, $q(P_{i+1}) \geq q(P_i)$. If also $q(P_{i+1}) \leq q(P_i) + \e$ then, since $f$ is decreasing and $P_{i+1}$ is $f(|P_i|)$-regular, we see that all three conditions of the theorem are satisfied with $P = P_i$ and $Q = P_{i+1}$ as the required partitions. Otherwise, we have $q(P_{i+1}) > q(P_i) + \epsilon$. However, since the mean-square density is bounded above by $1$, this can happen at most $\epsilon^{-1}$ times, concluding the proof.

It is not hard to see why this proof results in wowzer-type bounds. At each step, we are applying the regularity lemma to find a partition $P_{i+1}$ which is regular in the number of parts in the previous partition $P_i$. The bounds coming from the regularity lemma then imply that $|P_{i+1}| = T(f(|P_i|)^{-O(1)})$. But this iterated tower-type bound is essentially how we define the wowzer function.

That this is the correct behaviour for the bounds in the strong regularity lemma was proved independently by Conlon and Fox \cite{CF12} and by Kalyanasundaram and Shapira \cite{KSh12}, though both proofs use slightly different ideas and result in slightly different bounds. With the function $f : \mathbb{N} \rightarrow (0,1)$ taken to be $f(n) = \e/n$, the proof given in \cite{CF12} shows that the number of parts in the smaller partition $P$ may need to be as large as  wowzer in a power of $\epsilon^{-1}$, while that given in \cite{KSh12} proves that it must be at least wowzer in $\sqrt{\log \e^{-1}}$. 

The following easy corollary of the strong regularity lemma \cite{AFKS} is the key to proving the induced graph removal lemma.

\begin{lemma}\label{strongeasycor}
For each $0 < \epsilon < 1/3$ and decreasing function $f:\mathbb{N}\rightarrow
(0,1/3)$, there is $\delta'=\delta'(\epsilon,f)$ such that every graph $G=(V,E)$
with $|V| \geq \delta'^{-1}$ has an equitable partition $V=V_1 \cup \ldots \cup V_k$ and vertex subsets $W_i
\subset V_i$ such that $|W_i| \geq \delta' |V|$, each pair $(W_i,W_j)$ with
$1 \leq i \leq j \leq k$ is $f(k)$-regular and all but at most $\epsilon k^2$
pairs $1 \leq i \leq j \leq k$ satisfy $|d(V_i,V_j)-d(W_i,W_j)| \leq \epsilon$.
\end{lemma}

In fact, Lemma \ref{strongeasycor} is a little bit stronger than the original
version in \cite{AFKS} in that each set $W_i$ is $f(k)$-regular with itself.\footnote{This stronger version may be derived from an extra application of the regularity lemma within each of the pieces $W_i$, together with a suitable application of Ramsey's theorem. This is essentially the process carried out in \cite{AFKS}, though they do not state their final result in the same form as Lemma \ref{strongeasycor}.} The original version follows from the strong regularity lemma, applied with 
$$f'(k) = \min\left(f(k), \frac{\e}{4}, \frac{1}{2} \binom{k+2}{2}^{-1}\right),$$ by taking the
partition $V=V_1 \cup \ldots \cup V_k$ to be the partition $P$ in the
strong regularity lemma and the subset $W_i$ to be a
random part $V_{i, p} \subset V_i$ of the refinement $Q$ of
$P$ in the strong regularity lemma. Since $f'(k) \leq \frac{1}{2} \binom{k+2}{2}^{-1}$, it is
straightforward to check that all pairs $(W_i, W_j)$ are $f(k)$-regular with probability greater than
$\frac{1}{2}$. Moreover, the expected number of pairs with $1 \leq i < j \leq k$ for which 
$|d(V_i, V_j) - d(W_i, W_j)| > \e$ is at most 
$$\frac{\e}{4} \binom{k}{2} + \frac{\e}{4} \binom{k}{2} = \frac{\e}{2} \binom{k}{2}.$$
Here, the two $\frac{\e}{4}$ factors come from the definition of $f(1)$-closeness. The first factor comes from the fact that at most an $\frac{\e}{4}$-fraction of the pairs $(V_i, V_j)$ do not have good approximations while the second factor comes from the fact that for all other pairs there are at most an $\frac{\e}{4}$ fraction of pairs $(W_i, W_j)$ which do not satisfy $|d(V_i, V_j) - d(W_i, W_j)| \leq \e$. Therefore, by Markov's inequality, the probability that the number of bad pairs is greater than $\e \binom{k}{2}$ is
less than $\frac{1}{2}$. We therefore see that with positive probability there is a choice of $W_i$ satisfying the required weaker version of Lemma \ref{strongeasycor}.

If we assume the full strength of Lemma \ref{strongeasycor} as stated, that is, that each $W_i$ is also $f(k)$-regular with itself, it is easy to deduce the induced removal lemma.
Let $h = |V(H)|$ and take $f(k)=\frac{\epsilon^h}{4h}$. If there is a mapping
$\phi:V(H) \rightarrow \{1,\ldots,k\}$ such that for all adjacent vertices
$v,w$ of $H$, the edge density between $W_{\phi(v)}$ and $W_{\phi(w)}$ is at
least $\epsilon$ and for all distinct nonadjacent vertices $v,w$ of $H$, the
edge density between $W_{\phi(v)}$ and $W_{\phi(w)}$ is at most $1-\epsilon$,
then the following standard counting lemma (see, e.g., Lemma 3.2 in Alon, Fischer, Krivelevich and Szegedy \cite{AFKS} for a minor variant) 
shows that $G$ contains at least $\delta n^h$
induced copies of $H$, where $\delta=\frac{1}{h!}(\epsilon/4)^{{h \choose 2}}\delta'^h$. As with Lemma \ref{lem:AFKSusualcount}, the standard proof of this counting lemma uses a greedy embedding strategy.

\begin{lemma} \label{lem:AFKScount}
If $H$ is a graph with vertices $1,\ldots,h$ and $G$ is a graph with not
necessarily disjoint vertex subsets $W_1,\ldots,W_h$ such that every pair
$(W_i,W_j)$ with $1 \leq i < j \leq h$ is $\gamma$-regular with $\gamma \leq
\frac{\eta^h}{4h}$, $|W_i| \geq \gamma^{-1}$ for $1 \leq i \leq h$ and, for $1
\leq i < j \leq k$, $d(W_i,W_j)>\eta$ if $(i,j)$ is an edge of $H$ and
$d(W_i,W_j)<1-\eta$ otherwise, then $G$ contains at least
$\left(\frac{\eta}{4}\right)^{{h \choose 2}}|W_1| \times \cdots \times |W_h|$
induced copies of $H$ with the copy of vertex $i$ in $W_i$.
\end{lemma}

Hence, we may
assume that there is no such mapping $\phi$.
We then delete the edges between $V_i$ and $V_j$ if the edge density between $W_i$
and $W_j$ is less than $\epsilon$ and add the edges between $V_i$ and
$V_j$ if the density between $W_i$ and $W_j$ is more than $1-\epsilon$. The
total number of edges added or removed is at most $5\epsilon n^2$ and no
induced copy of $H$ remains. Replacing $\epsilon$ by $\epsilon/8$ in the above argument gives the induced removal lemma. 

\subsection{An improved bound} \label{sec:improvedinduced}

The main goal of this section is to prove Theorem \ref{thm:inducedimproval}, which gives a bound on $\delta^{-1}$ which is a tower in $h$ of height polynomial in $\epsilon^{-1}$. We in fact prove the key corollary of the strong regularity lemma, Lemma \ref{strongeasycor}, with a tower-type bound.
This is sufficient to prove the desired tower-type bound for the induced graph removal lemma. 

As in Section \ref{sec:improvedremoval}, the key idea will be to take a weak variant of Szemer\'edi's regularity lemma and iterate it. The particular variant we will use, due to Duke, Lefmann and R\"odl \cite{DLR}, was originally used by them to derive a fast approximation algorithm for the number of copies of a fixed graph in a large graph.

A {\it $k$-cylinder} (or cylinder for short) in a graph $G$ is a product of $k$ vertex subsets. 
Given a $k$-partite graph $G=(V,E)$ with $k$-partition $V=V_1 \cup \ldots \cup
V_k$, we will consider a partition $\mathcal{K}$ of the cylinder $V_1 \times
\cdots \times V_k$ into cylinders $K=W_1 \times \cdots \times W_k$, $W_i
\subset V_i$ for $i=1,\ldots,k$ and we let $V_i(K)=W_i$.  We say that a cylinder is {\it $\epsilon$-regular} if all ${k \choose 2}$
pairs of subsets $(W_i,W_j)$, $1 \leq i < j \leq k$, are $\epsilon$-regular.
The partition $\mathcal{K}$ is {\it $\epsilon$-regular} if all but an
$\epsilon$-fraction of the $k$-tuples $(v_1,\ldots,v_k) \in V_1 \times \cdots
\times V_k$ are in $\epsilon$-regular cylinders in the partition $\mathcal{K}$.

The weak regularity lemma of Duke, Lefmann and R\"odl \cite{DLR} is now as
follows. Note that, like the Frieze-Kannan weak regularity lemma, it has only a
single-exponential bound on the number of parts. We will sometimes refer to this 
lemma as the cylinder regularity lemma.

\begin{lemma}\label{dukelefrod}
Let $0<\epsilon<1/2$ and $\beta=\beta(\epsilon)=\epsilon^{k^2\epsilon^{-5}}$.
Suppose $G=(V, E)$ is a $k$-partite graph with $k$-partition $V=V_1 \cup \ldots
\cup V_k$. Then there exists an $\epsilon$-regular partition $\mathcal{K}$ of
$V_1 \times \cdots \times V_k$ into at most $\beta^{-1}$ parts such that, for
each $K \in \mathcal{K}$ and $1 \leq i \leq k$, $|V_i(K)| \geq
\beta|V_i|$.
\end{lemma}

We would now like to iterate this lemma to get a stronger version, the strong cylinder regularity lemma. Like Lemmas \ref{generalreg} and \ref{strongreg}, this will yield two closely related cylinder partitions $P$ and $Q$ with $P$ regular and $Q$ regular in a function of $|P|$. To state the lemma, we first strengthen the definition of regular cylinders so that pieces are also regular with themselves.

A $k$-cylinder $W_1 \times \cdots \times W_k$ is {\it strongly
$\epsilon$-regular} if all pairs $(W_i,W_j)$ with $1 \leq i,j \leq k$ are
$\epsilon$-regular. A partition $\mathcal{K}$ of $V_1 \times \cdots \times V_k$
into cylinders is {\it strongly $\epsilon$-regular} if all but $\epsilon|V_1| \times
\cdots \times |V_k|$ of the $k$-tuples $(v_1,\ldots,v_k) \in V_1 \times \cdots \times
V_k$ are contained in strongly $\epsilon$-regular cylinders $K \in
\mathcal{K}$.

We now state the strong cylinder regularity lemma. Here $t_i(x)$ is a variant of the tower function defined by $t_0(x) = x$ and $t_{i+1}(x) = 2^{t_i(x)}$. Also, given a cylinder partition $\mathcal{K}$, $Q(\mathcal{K})$ is the coarsest vertex partition such that every set $V_i(K)$ with $i \in [k]$ and $K \in \mathcal{K}$ is the union of parts of $Q(\mathcal{K})$.

\begin{lemma} \label{scr2}
For $0<\epsilon<1/3$, positive integer $s$, and decreasing function $f:\mathbb{N} \rightarrow
(0,\epsilon]$, there is $S=S(\epsilon,s,f)$ such that the following holds. For every
graph $G$, there is an integer $s \leq k \leq S$, an equitable  partition $P:V=V_1 \cup
\ldots \cup V_k$ and a strongly $f(k)$-regular partition $\mathcal{K}$ of the
cylinder $V_1 \times \cdots \times V_k$ into cylinders satisfying that the
partition $Q=Q(\mathcal{K})$ of $V$ has at most $S$ parts and $q(Q) \leq
q(P)+\epsilon$. Furthermore, there is an absolute constant $c$ such that
letting $s_1=s$ and $s_{i+1}=t_4\left(\left(s_i/f(s_i)\right)^c\right)$, we may
take $S=s_{\ell}$ with $\ell=2\epsilon^{-1}+1$.
\end{lemma}

To prove this lemma, we need to find a way to guarantee that the parts of the cylinder partition are regular with themselves as required in the definition of strong cylinder regularity. For a graph $G=(V,E)$, a vertex subset $U \subset V$ is {\it $\epsilon$-regular} if the pair $(U,U)$ is $\epsilon$-regular. The following
lemma, which demonstrates that any graph contains a large vertex subset which is
$\epsilon$-regular, is the first step. 

\begin{lemma}\label{oneepsilonregularsubset}
For each $0<\epsilon<1/2$, let
$\delta=\delta(\epsilon)=2^{-\epsilon^{-(10/\epsilon)^{4}}}$. Every graph
$G=(V,E)$ contains an $\epsilon$-regular vertex subset $U$ with $|U| \geq
\delta |V|$.
\end{lemma}

One way to prove this lemma is to first find a large collection $C$ of disjoint subsets of equal order which are pairwise $\alpha$-regular with $\alpha=(\epsilon/3)^2$. This can be done by an application of Szemer\'edi's regularity lemma and Tur\'an's theorem, but then the bounds are quite weak. Instead, one can  easily deduce this from Lemma \ref{dukelefrod}. A further application of Ramsey's theorem allows one to get a subcollection $C'$ of size $s\geq 2\alpha^{-1}$ such that the edge density between each pair of distinct subsets in $C'$ lies in an interval of length at most $\alpha$. The union of the sets in $C'$ is then an $\epsilon$-regular subset of the desired order. 

It is crucial in this lemma that $\delta^{-1}$ be of bounded tower height in $\epsilon^{-1}$. While our bound gives a double exponential dependence, we suspect that the truth is more likely to be a single exponential. We leave this as an open problem.

Repeated applications of Lemma \ref{oneepsilonregularsubset} allow us to pull out large, regular subsets until a small fraction of vertices remain. By distributing the remaining vertices amongst these subsets, we only slightly weaken their regularity, while giving a partition of any graph into large parts each of which is $\e$-regular with itself. This will be sufficient for our purposes.

\begin{lemma}\label{epsdeltaone}
For each $0<\epsilon<1/2$, let
$\delta=\delta(\epsilon)=2^{-\epsilon^{-(20/\epsilon)^{4}}}$. Every graph
$G=(V,E)$ has a vertex partition $V=V_1 \cup \ldots \cup V_k$ such that for
each $i$, $1 \leq i \leq k$, $|V_i| \geq \delta|V|$ and $V_i$ is an
$\epsilon$-regular set.
\end{lemma}

We are now ready to prove the strong cylinder regularity lemma.
\vspace{3mm}

{\bf Proof of Lemma \ref{scr2}:}\hspace{1mm}
We may assume $|V| \geq S$, as otherwise we can let $P$ and $Q$ be the trivial
partitions into singletons, and it is easy to see the lemma holds.
We will define a sequence of partitions $P_1,P_2,\ldots$ of equitable partitions,
with $P_{j+1}$ a refinement of $P_j$ and $q(P_{j+1}) > q(P_j)+\epsilon/2$. Let
$P_1$ be an arbitrary equitable partition of $V$ consisting of $s_1=s$ parts. Suppose we
have already found an equitable partition $P_j:V=V_1 \cup \ldots \cup V_k$ with
$k \leq s_{j}$.

Let $\beta(x,\ell)=x^{\ell^2x^{-5}}$ as in Lemma \ref{dukelefrod} and
$\delta(x)=2^{-x^{-(20/x)^4}}$ as in Lemma \ref{epsdeltaone}. We apply Lemma
\ref{epsdeltaone} to each part $V_i$ of the partition $P_j$ to get a partition
of each part $V_i=V_{i1} \cup \ldots \cup V_{ih_i}$ of $P_i$ into parts each of
cardinality at least $\delta|V_i|$, where $\delta=\delta(\gamma)$ and
$\gamma=f(k) \cdot \beta$ with $\beta=\beta(f(k),k)$, such that each part $V_{ih}$ is
$\gamma$-regular. Note that $\delta^{-1}$ is at most triple-exponential in a
polynomial in $k/f(k)$. For each $k$-tuple $\ell=(\ell_1,\ldots,\ell_k) \in
[h_1] \times \cdots \times [h_k]$,  by Lemma \ref{dukelefrod} there is an
$f(k)$-regular partition $\mathcal{K}_{\ell}$ of the cylinder $V_{1\ell_1}
\times \cdots \times V_{k\ell_k}$ into at most $\beta^{-1}$ cylinders such
that, for each $K \in \mathcal{K}_{\ell}$, $|V_{i\ell_i}(K)| \geq \beta
|V_{i\ell_i}|$.   The union of the $\mathcal{K}_{\ell}$ forms a partition
$\mathcal{K}$  of $V_1 \times \cdots \times V_k$ which is strongly
$f(k)$-regular.

Recall that $Q=Q(\mathcal{K})$ is the partition of $V$ which is the common refinement of all
parts $V_i(K)$ with $i \in [k]$ and $K \in \mathcal{K}$. The number of parts of
$\mathcal{K}$ is at most $\delta^{-k}\beta^{-1}$ and hence the number of parts
of $Q$ is at most $k2^{1/(\delta^k \beta)}$. Thus, the number of parts of $Q$
is at most quadruple-exponential in a polynomial in $k/f(k)$. Let $P_{j+1}$ be
an equitable partition into $4\epsilon^{-1}|Q|$ parts with
$q(P_{j+1}) \geq q(Q)-\frac{\epsilon}{2}$, which exists by Lemma
\ref{makeequip}. Hence, there is an absolute constant $c$ such that
$$|P_{j+1}| \leq t_4\left((k/f(k))^c\right) \leq s_{j+1}.$$

If $q(Q) \leq q(P_j)+\epsilon$, then we may take $P=P_j$ and
$Q=Q(\mathcal{K})$, and these partitions satisfy the desired properties.
Otherwise, $q(P_{j+1}) \geq q(Q)-\frac{\epsilon}{2} >
q(P_j)+\frac{\epsilon}{2}$, and we continue the sequence of partitions. Since
$q(P_1) \geq 0$ and the mean-square density goes up by more than $\epsilon/2$
at each step and is always at most $1$, this process must stop within
$2/\epsilon$ steps, and we obtain the desired partitions.
\hfill$\Box$\medskip

Let $G=(V,E)$, $P:V=V_1 \cup \ldots \cup V_k$ be an equipartition and
$\mathcal{K}$ be a partition of the cylinder $V_1 \times \cdots \times V_k$
into cylinders. For $K=W_1 \times \cdots \times W_k \in \mathcal{K}$, define
the density $d(K)=\frac{|W_1| \times \cdots \times |W_k|}{|V_1| \times \cdots
\times |V_k|}$. The cylinder $K$ is {\it $\epsilon$-close} to $P$ if
$\left|d(W_i,W_j)-d(V_i,V_j)\right| \leq \epsilon$ for all but at most $\epsilon k^2$ pairs $1 \leq i \not = j \leq k$. The cylinder partition $\mathcal{K}$ is {\it $\epsilon$-close} to $P$ if  $\sum d(K) \leq \epsilon$, where the sum is over all $K \in \mathcal{K}$ that are not
$\epsilon$-close to $P$. As with the definition of closeness used in the strong regularity lemma, this definition is closely related to the condition that $q(Q) \leq q(P) + \e$, where here $Q = Q(\mathcal{K})$. 

The connection we shall need to prove Lemma \ref{strongeasycor} is contained in the following statement.

\begin{lemma}\label{cylinderclose}
Let $G=(V,E)$ and $P:V=V_1 \cup \ldots \cup V_k$ be an equipartition with $k \geq 2\epsilon^{-1}$ and $|V| \geq 4k\epsilon^{-1}$. Let $\mathcal{K}$ be a partition of the
cylinder $V_1 \times \cdots \times V_k$ into cylinders. If $Q=Q(\mathcal{K})$
satisfies $q(Q) \leq q(P)+\epsilon$, then $\mathcal{K}$ is $(2\epsilon)^{1/4}$-close to $P$.
\end{lemma}
\begin{proof}
It will be helpful to assume that all parts of the equipartition $P$ have equal size - this affects the calculations only slightly. It will also be helpful to introduce a slight variant of the mean-square density as follows. Let $q'(P)=\sum_{i<j}d^2(V_i,V_j)p_{ij}$, where $p_{ij}=|V_i||V_j|/\sum_{i<j} |V_i||V_j|$. 
Thus, $q'(P)$ is the mean of the square densities between the pairs of distinct parts. It is easy to check that $q'(P)$ is close to $q(P)$. Indeed, we have $q'(P)-q(P)=
\frac{1}{k}\left(q'(P)-\bar{q}\right)$, where $\bar{q}=\sum_{i=1}^k d^2(V_i)/k$ is the average of the square densities inside the parts. Hence, $|q'(P)-q(P)| \leq  \frac{1}{k}$. 
We similarly have $|q'(Q)-q(Q)| \leq \frac{1}{k}$. Let $$q(\mathcal{K})={k \choose 2}^{-1}\sum_{i<j}\sum_{K \in \mathcal{K}}d^2(V_i(K),V_j(K))d(K).$$ We have the following equalities 
\begin{eqnarray*} q(\mathcal{K})-q'(P) & = & {k \choose 2}^{-1}\sum_{i<j}\sum_{K \in \mathcal{K}}\left(d^2(V_i(K),V_j(K))-d^2(V_i,V_j)\right)d(K) \\  & = &
{k \choose 2}^{-1}\sum_{i<j}\sum_{K \in \mathcal{K}}\left(d(V_i(K),V_j(K))-d(V_i,V_j)\right)^2 d(K),\end{eqnarray*}
where the last equality uses the identity $d(V_i,V_j)=\sum_{K \in \mathcal{K}}d(V_i(K),V_j(K))d(K)$. 
This equality shows that $q(\mathcal{K}) \geq q'(P)$ as it expresses their difference as a sum of nonnegative terms. Furthermore, it shows that if $\mathcal{K}$ is not $\beta$-close to $P$, 
then $q(\mathcal{K}) \geq q(P)+{k \choose 2}^{-1} \cdot \frac{\beta k^2}{2} \cdot  \beta^2 \cdot \beta \geq q(P)+\beta^4$. In particular, if $q(\mathcal{K}) \leq q'(P)+2\epsilon$, then $\mathcal{K}$ is $(2\epsilon)^{1/4}$-close to $P$. So assume for contradiction that $q(\mathcal{K}) > q'(P)+2\epsilon$. 

A similar equality implies $q'(Q) \geq q(\mathcal{K})$.  We therefore have 
\begin{eqnarray*} q(Q)-q(P) & = & \left(q(Q)-q'(Q)\right)+\left(q'(Q)-q(\mathcal{K})\right)+\left(q(\mathcal{K})-q'(P)\right)+\left(q'(P)-q(P)\right)
\\ & \geq & -\frac{1}{k} + 0 + \left(q(\mathcal{K})-q'(P)\right) - \frac{1}{k} 
\\ & > & \epsilon,   
\end{eqnarray*} 
contradicting the assumption of Lemma \ref{cylinderclose} and completing the proof.
\end{proof} 

With this in hand, we can readily deduce a tower-type bound for Lemma \ref{strongeasycor}.

\begin{lemma} \label{easycor2}
For each $0 < \epsilon < 1/3$ and decreasing function $f:\mathbb{N}\rightarrow
(0,\epsilon]$, there is $\delta'=\delta'(\epsilon,f)$ such that every graph $G=(V,E)$ with $|V| \geq \delta'^{-1}$ 
has an equitable partition $V=V_1 \cup \ldots \cup V_k$ and vertex subsets $W_i
\subset V_i$ such that $|W_i| \geq \delta' |V|$, each pair $(W_i,W_j)$ with
$1 \leq i \leq j \leq k$ is $f(k)$-regular and all but at most $\epsilon k^2$
pairs $1 \leq i \leq j \leq k$ satisfy $|d(V_i,V_j)-d(W_i,W_j)| \leq \epsilon$.  Furthermore, we may take $\delta'=\frac{1}{8S^2}$, where $S=S(\frac{\epsilon^4}{2},s,f)$ is defined as in Lemma \ref{scr2} and $s=2\epsilon^{-1}$.
\end{lemma}
\begin{proof} Let $\alpha=\frac{\epsilon^4}{2}$, $s=2\epsilon^{-1}$, and $\delta'=\frac{1}{8S^2}$,
where $S=S(\alpha,s,f)$ is as in Lemma \ref{scr2}. We apply Lemma \ref{scr2} with
$\alpha$ in place of $\epsilon$. We get an equipartition $P:V=V_1 \cup \ldots
\cup V_k$ with $s \leq k \leq S$ and a strongly $f(k)$-regular partition $\mathcal{K}$ of $V_1 \times
\cdots \times V_k$ into cylinders such that the refinement $Q=Q(\mathcal{K})$
of $P$ has at most $S=S(\alpha, s, f)$ parts and satisfies $q(Q) \leq q(P)+\alpha$. 
Since $|V| \geq \delta'^{-1}=8S^2$, and $P$ is an equipartition into $k \leq S$ parts, the cardinality of each part $V_i \in P$ satisfies $|V_i| \geq \frac{|V|}{2S}$. By Lemma \ref{cylinderclose},  as $(2 \alpha)^{1/4}=\epsilon$, the cylinder
partition $\mathcal{K}$ is $\epsilon$-close to $P$. Hence, at most an
$\epsilon$-fraction of the $k$-tuples $(v_1,\ldots,v_k) \in V_1 \times \cdots
\times V_k$ belong to parts $K=W_1 \times \cdots \times W_k$ of $\mathcal{K}$
that are not $\epsilon$-close to $P$. Since $Q(\mathcal{K})$ has at most $S$
parts, the fraction of $k$-tuples $(v_1,\ldots,v_k) \in V_1 \times \cdots
\times V_k$ that belong to parts $K=W_1 \times \cdots \times W_k$ of
$\mathcal{K}$ with $|W_i|<\frac{1}{4S} |V_i|$ for at least one $i \in [k]$ is at
most $\frac{1}{4S}  \cdot S = \frac{1}{4}$. Therefore, at least a fraction $1-f(k)-\epsilon-\frac{1}{4}>0$
of the $k$-tuples  $(v_1,\ldots,v_k) \in V_1 \times \cdots \times V_k$ belong
to parts $K=W_1 \times \cdots \times W_k$ of $\mathcal{K}$ satisfying $K$ is
strongly $f(k)$-regular,  $|W_i| \geq \frac{1}{4S} |V_i| \geq \delta'|V|$ for $i \in [k]$ and $K$
is $\epsilon$-close to $P$.  Since a positive fraction of the $k$-tuples belong
to such $K$, there is at least one such $K$. This $K$ has the desired
properties. Indeed, the number of pairs $1 \leq i \not = j \leq k$ for which $|d(W_i,W_j)-d(V_i,V_j)| >
\epsilon$ is at most $\epsilon k^2$ and hence the number of pairs $1 \leq i \leq j \leq k$ for 
which $|d(W_i,W_j)-d(V_i,V_j)| > \epsilon$ is at most $\epsilon k^2/2+k \leq \epsilon k^2$. This completes the proof. 
\end{proof}

By using the induced counting lemma, Lemma \ref{lem:AFKScount}, we may now conclude the proof as in Section \ref{sec:usualinduced} to obtain the 
following quantitative version of Theorem \ref{thm:inducedremoval}.

\begin{theorem}
There exists a constant $c$ such that, for any graph $H$ on $h$ vertices and $0 < \epsilon< 1/2$, if $\delta^{-1} =
t_j(h)$, where $j= c \epsilon^{-4}$, then any graph $G$ on $n$ vertices with
at most $\delta n^h$ induced copies of $H$ may be made induced $H$-free by adding and/or deleting at most $\epsilon n^2$ edges.
\end{theorem}

\subsection{Infinite removal lemma} \label{sec:infinite}

In order to characterize the natural graph properties which are testable, the induced removal lemma was extended by Alon and Shapira \cite{AlSh08} to the following infinite version. For a family $\mathcal{H}$ of graphs, a graph $G$ is induced $\mathcal{H}$-free if $G$ does not contain any graph $H$ in $\mathcal{H}$.
 
\begin{theorem}\label{infiniteremoval} For every (possibly infinite) family of graphs $\mathcal{H}$ and $\epsilon>0$, there are $n_0$, $h_0$, and $\delta$ such that the following holds. If a graph $G=(V,E)$ on $n \geq n_0$ vertices has at most $\delta n^{h}$ induced copies of each graph $H \in \mathcal{H}$ on $h \leq h_0$ vertices, then $G$ can be made induced $\mathcal{H}$-free by adding and/or deleting at most $\epsilon n^2$ edges.
\end{theorem}
\begin{proof} The proof is a natural extension of the proof of the induced removal lemma and similarly uses the key corollary, Lemma \ref{strongeasycor}, of the strong regularity lemma. The main new idea is to pick an appropriate function $f$ to apply Lemma \ref{strongeasycor}. The choice of the function $f$ will depend heavily on the family $\mathcal{H}$. 

For a graph $H$ and an edge-coloring $c$ of the edges of the complete graph with loops $R$ on $[k]$ with colors white, black and grey, we write $H \rightarrow_c R$ if there is a mapping $\phi:V(H) \rightarrow [k]$ such that for each edge $(u,v)$ of $H$ we have that $c(\phi(u),\phi(v))$ is black or grey and for each pair $(u,v)$ of distinct vertices of $H$ which do not form an edge we have that $c(\phi(u),\phi(v))$ is white or grey. We write $H \not \rightarrow_c R$ if $H \rightarrow_c R$  does not hold. 

Let $P:V=V_1 \cup \ldots \cup V_k$ be a vertex partition of $G$. A key observation is that if we {\it round} $G$ by the partition $P$ and the coloring $c$ to obtain a graph $G'$ on the same vertex set as $G$ by adding edges to make $(V_i,V_j)$ complete if $(i,j)$ is black, deleting edges to make $(V_i,V_j)$ empty if $(i,j)$ is white and we have that $H \not \rightarrow_c R$, then $G'$ does not contain $H$ as an induced subgraph.  

For any (possibly infinite) family of graphs $\mathcal{H}$ and any integer $r$, let $\mathcal{H}_r$ be the following set of colored complete graphs with loops: a colored complete graph with loops $R$ belongs to $\mathcal{H}_r$ if and only if it has at most $r$ vertices and there is at least one $H \in \mathcal{H}$ such that $H \rightarrow_c R$. For any family $\mathcal{H}$ of graphs and integer $r$ for which $\mathcal{H}_r \not = \emptyset$, let $$\Psi_{\mathcal{H}}(r)=\max_{R \in \mathcal{H}_r} \, \min_{H \in \mathcal{H}:H \rightarrow_c R} |V(H)|.$$
If $\mathcal{H}_r =\emptyset$, define $\Psi_{\mathcal{H}}(r)=1$. Note that $\Psi_{\mathcal{H}}(r)$ is a monotonically increasing function of $r$. 
Let $$f(r)=\frac{\epsilon^{\Psi_{\mathcal{H}}(r)}}{4\Psi_{\mathcal{H}}(r)}.$$ 
Note that the function $f$ only depends on $\epsilon$ and $\mathcal{H}$. 

Let $\delta'=\delta'(\epsilon,f)$ be as in Lemma \ref{strongeasycor}, which only depends on $\epsilon$ and $\mathcal{H}$. Also let $k_0=2\delta'^{-1}$, $h_0=\Psi_{\mathcal{H}}(k_0)$, $n_0=1/\left(\delta' f(k_0)\right)$ and $\delta=\frac{1}{h_0!}(\epsilon/4)^{h_0^2}\delta'^{\, h_0}$. We have that $k_0$, $h_0$, $n_0$ and $\delta>0$ only depend on $\epsilon$ and $\mathcal{H}$. By assumption, $G$ has $n \geq n_0$ vertices. 

We apply Lemma \ref{strongeasycor} to $G$. We get an equitable vertex partition $P:V=V_1 \cup \ldots \cup V_k$ of $G$ and subsets $W_i \subset V_i$ with $|W_i| \geq \delta'|V|$ such that, for $1 \leq i \leq j \leq k$, the pair $(W_i,W_j)$ is $f(k)$-regular and all but at most $\epsilon k^2$ pairs $1 \leq i \leq j \leq k$ satisfy $|d(V_i,V_j)-d(W_i,W_j)| \leq \epsilon$. As $\delta'|V| \leq |W_i| \leq |V_i| \leq 2n/k$, we have $k \leq 2\delta'^{-1}  \leq k_0$.

Consider the coloring $c$ of the complete graph with loops $R$ on $[k]$ where a pair $(i,j)$ of vertices is black if $d(W_i,W_j) \geq 1-\epsilon$, white if $d(W_i,W_j)\leq \epsilon$ and grey if $\epsilon<d(W_i,W_j)<1-\epsilon$. Suppose, for the sake of contradiction, that there is a graph $H$ with $H \rightarrow_c R$. From the definition of $\Psi$, there is a graph $H$ on $h \leq \Psi_{\mathcal{H}}(k)$ vertices with $H \rightarrow_c  R$.  As $k \leq k_0$, the number of vertices of $H$ satisfies $h \leq h_0$.  As each pair $(W_i,W_j)$ is $f(k)$-regular and $|W_i| \geq \delta'|V| \geq f(k)^{-1}$, applying the induced counting lemma, Lemma \ref{lem:AFKScount}, with $\gamma=f(k)$, we get at least $$\frac{1}{h!}\left(\frac{\epsilon}{4}\right)^{h \choose 2}(\delta'|V|)^h \geq \delta n^h$$ induced copies of $H$ in $G$, contradicting the supposition of the theorem. Thus, there is no graph $H$ with $H \rightarrow_c R$. 

We round the graph $G$ by the partition $P$ and the coloring $c$ as described earlier in the proof to obtain a graph $G'$. By the key observation, for each graph $H$ with $H \not \rightarrow_c R$, the graph $G'$ does not contain $H$ as an induced subgraph. Hence, $G'$ is induced $\mathcal{H}$-free.  

Moreover, not many edges were changed from $G$ to obtain $G'$. Indeed, as there are at most $\epsilon k^2$ pairs $1 \leq i \leq j \leq k$ which satisfy $|d(V_i,V_j)-d(W_i,W_j)| > \epsilon$,  the number of edge modifications made between such pairs is at most $\epsilon k^2 \cdot (2n/k)^2=4\epsilon n^2$. Between the other pairs we have made at most $2\epsilon {n \choose 2} \leq \epsilon n^2$ edge modifications. In total, at most $5\epsilon n^2$ edge modifications were made to obtain $G'$ from $G$. Replacing $\epsilon$ by $\epsilon/5$ in the above argument completes the proof.   
\end{proof}

\section{Arithmetic removal} \label{sec:arithmeticremoval}

The notion of arithmetic removal was introduced by Green \cite{G05}. By establishing an appropriate variant of the regularity lemma in the context of abelian groups, he proved the following result. 

\begin{theorem} \label{thm:greenremoval}
For any natural number $k \geq 3$ and any $\e > 0$, there exists $\d > 0$ such that if $G$ is an abelian group of order $n$ and $A_1, \dots, A_k$ are subsets of $G$ such that there are at most $\d n^{k-1}$ solutions to the equation $a_1 +  a_2 + \dots +  a_k = 0$ with $a_i \in A_i$ for all $i$ then it is possible to remove at most $\e n$ elements from each set $A_i$ to form sets $A'_i$ so that there are no solutions to the equation $a'_1 + a'_2 + \dots + a'_k = 0$ with $a'_i \in A'_i$ for all $i$.
\end{theorem}

It is an exercise to show that Green's result implies Roth's theorem. While Green's proof of this result relied on Fourier analytic techniques, an alternative proof was found by Kr\'al', Serra, and Vena \cite{KSV09}, who showed that the following more general result follows from an elegant reduction to the removal lemma in directed graphs. 

\begin{theorem} \label{thm:oneeqnremoval}
For any natural number $k \geq 3$ and any $\e > 0$, there exists $\d > 0$ such that if $G$ is a group of order $n$, $g \in G$ and $A_1, \dots, A_k$ are subsets of $G$ such that there are at most $\d n^{k-1}$ solutions to the equation $a_1 a_2 \cdots a_k = g$ with $a_i \in A_i$ for all $i$ then it is possible to remove at most $\e n$ elements from each set $A_i$ to form sets $A'_i$ so that there are no solutions to the equation $a'_1 a'_2 \cdots a'_k = g$ with $a'_i \in A'_i$ for all $i$.
\end{theorem}

This is stronger than Theorem \ref{thm:greenremoval} in two ways. Firstly, it applies to all groups and not just to abelian groups. Secondly, it applies to non-homogeneous equations, that is, $a_1 a_2 \cdots a_k = g$ for a general $g$, whereas Green only treats the homogeneous case where $g = 1$. To give some idea of their proof, we will need the following definition.

A directed graph is a graph where each edge has been given a direction. Formally, the edge set may be thought of as a collection of ordered pairs. We will always assume that the directed graph has no loops and does not contain parallel directed edges, though we do allow anti-parallel edges, that is, both the edge $\vec{uv}$ and the edge $\vec{vu}$. The following analogue of the graph removal lemma for directed graphs was proved by Alon and Shapira \cite{AlSh04} as part of their study of property testing in directed graphs.

\begin{theorem} \label{thm:directedremoval}
For any directed graph $H$ and any $\e > 0$, there exists $\d > 0$ such that any directed graph on $n$ vertices which contains at most $\d n^{v(H)}$ copies of $H$ may be made $H$-free by removing at most $\e n^2$ edges. 
\end{theorem}

We will show how to prove Theorem \ref{thm:oneeqnremoval} with $g = 1$ using Theorem \ref{thm:directedremoval}. Suppose that $G$ is a group of order $n$ and $A_1, \dots, A_k$ are subsets of $G$ such that there are at most $\d n^{k-1}$ solutions to the equation $a_1 a_2 \cdots a_k = 1$ with $a_i \in A_i$ for all $i$. Consider the auxiliary directed graph $\Gamma$ whose vertex set is $G \times \{1, 2, \dots, k\}$. We place an edge from $(x, i)$ to $(y, i+1)$, where addition is taken modulo $k$, if there exists $a_i \in A_i$ such that $x a_i = y$. It is easy to see that any directed cycle in $\Gamma$ corresponds to a solution of the equation $a_1 a_2 \cdots a_k = 1$. Moreover, every such solution will result in $n$ different directed cycles in $\Gamma$, namely, those with vertices $(x, 1), (x a_1, 2), (xa_1a_2,3),\dots, (x a_1 \cdots a_{k-1}, k)$.

Since $G$ has at most $\d n^{k-1}$ solutions to $a_1 a_2 \cdots a_k = 1$, this implies that there are at most $\d n^k$ directed cycles in $\Gamma$. By Theorem \ref{thm:directedremoval}, for an appropriately chosen $\d$, we may therefore remove at most $\frac{\e}{k} n^2$ edges to make it free of directed cycles of length $k$. In $A_i$, we now remove the element $a_i$ if at least $\frac{n}{k}$ edges of the form $(x, i) (x a_i, i+1)$ have been removed. Note that this results in us removing at most $\e n$ elements from each $A_i$. Suppose now that the remaining sets $A'_i$ are such that there is a solution $a'_1 a'_2 \dots a'_k = 1$ with $a'_i \in A'_i$ for all $i$. Then, as above, there are at least $n$ cycles $(x, 1), (x a'_1, 2), \dots, (x a'_1 \cdots a'_{k-1}, k)$ corresponding to this solution. Since we must have removed one edge from each of these cycles, we must have removed at least $\frac{n}{k}$ edges of the form $(y, i) (y a'_i, i+1)$ for some $i$. But this implies that $a'_i \not\in A'_i$, yielding the required contradiction.

It was observed by Fox \cite{F11} that $\d^{-1}$ in Theorem \ref{thm:directedremoval} may, like the graph removal lemma, be taken to be at most a tower of twos of height logarithmic in $\epsilon^{-1}$. This may in turn be used to give a similar bound for $\d^{-1}$ in Theorem \ref{thm:oneeqnremoval}.

In \cite{KSV09}, Kr\'al', Serra and Vena also showed how to prove a removal lemma for systems of equations which are graph representable, in the sense that they can be put in a natural correspondence with a directed graph. An example of such a system is 
\begin{align*}
x_1 x_2 x_4^{-1} x_3^{-1} & = 1\\
x_1 x_2 x_5^{-1} & = 1.
\end{align*}
This idea of associating a system of linear equations with a directed graph representation was extended to hypergraphs independently by Kr\'al', Serra and Vena \cite{KSV12} and by Shapira \cite{Sh09, Sh10} in order to prove the following theorem (some partial results had been obtained earlier by Kr\'al', Serra and Vena \cite{KSV08}, Szegedy \cite{Sz10} and Candela \cite{Can09}). 

\begin{theorem} \label{thm:systemremoval}
For any natural numbers $k$ and $\ell$ and any $\e > 0$, there exists $\d > 0$ such that if $F$ is the field of size $n$, $M$ is an $\ell \times k$ matrix with coefficients in $F$, $b \in F^\ell$ and $A_1, \dots, A_k$ are subsets of $F$ such that there are at most $\d n^{k - \ell}$ solutions $a = (a_1, \dots, a_k)$ of the system $Ma = b$ then it is possible to remove at most $\e n$ elements from each set $A_i$ to form sets $A'_i$ so that there are no solutions $a' = (a'_1, \dots, a'_k)$ to the equation $M a' = b$ with $a'_i \in A'_i$ for all $i$.
\end{theorem}

An easy application of this result shows that a removal lemma for systems of linear equations holds in the set $[n]$, confirming a conjecture of Green \cite{G05}. We remark that this result easily implies 
Szemer\'edi's theorem. Both proofs use a colored variant of the hypergraph removal lemma due to Austin and Tao \cite{AT10}, though the representations which they use to transfer the problem to hypergraphs are different. 

It would be interesting to know whether an analogous statement holds for all groups. A partial extension of these results to abelian groups is proved in \cite{KSV122} (see also \cite{Sz10}) but already in this case there are technical difficulties which do not arise for finite fields. 

\section{Sparse removal} \label{sec:sparseremoval}

Given graphs $\Gamma$ and $H$, let $N_H(\Gamma)$ be the number of copies of $H$ in $\Gamma$. A possible generalization of the graph removal lemma, which corresponds to the case $\Gamma = K_n$, could state that if $G$ is a subgraph of $\Gamma$ with $N_H(G) \leq \d N_H(\Gamma)$ then $G$ may be made $H$-free by deleting at most $\e e(\Gamma)$ edges. Unfortunately, this is too much to hope in general. However, if the graph $\Gamma$ is sufficiently well-behaved, such an extension does hold. We will discuss two such results here. 

\subsection{Removal in random graphs} \label{sec:random}

The binomial random graph $G_{n,p}$ is formed by taking $n$ vertices and considering each pair of vertices in turn, choosing each connecting edge to be in the graph independently with probability $p$. These graphs were introduced by Erd\H{o}s and R\'enyi \cite{ER59, ER60} in the late fifties\footnote{The notion was also introduced independently by several other authors at about the same time but, quoting Bollob\'as \cite{B01}, ``Erd\H{o}s and R\'enyi introduced the methods which underlie the probabilistic treatment of random graphs. The other authors were all concerned with enumeration problems and their techniques were essentially deterministic."} and their study has grown enormously since then (see, for example, the monographs \cite{B01, JLR00}). 

Usually, one is interested in finding a threshold function $p^* := p^*(n)$ where the probability that the random graph $G_{n,p}$ has a particular property $\mathcal{P}$ changes from $o(1)$ to $1-o(1)$ as we pass from random graphs chosen with probability $p \ll p^*$ to those chosen with probability $p \gg p^*$. For example, the threshold for the random graph to be connected is at $p^*(n) = \frac{\ln n}{n}$. 

One theme that has received a lot of attention in recent years is the question of determining thresholds for the appearance of certain combinatorial properties. One well-studied example is the Ramsey property. Given a graph $H$ and a natural number $r \geq 2$, we say that a graph $G$ is {\it $(H,r)$-Ramsey} if in any $r$-coloring of the edges of $G$ there is guaranteed to be a monochromatic copy of $H$. Ramsey's theorem \cite{R30} is itself the statement that $K_n$ is $(H, r)$-Ramsey for $n$ sufficiently large. The following celebrated result of R\"odl and Ruci\'nski \cite{RR93, RR95} from 1995 (see also \cite{JLR00}, Chapter 8) determines the threshold for the appearance of the Ramsey property in random graphs.

\begin{theorem} \label{thm:RR}
For any graph $H$ that is not a forest consisting of stars and paths of length $3$ and every positive integer~$r \geq 2$, there exist constants $c, C > 0$ such that 
\[
\lim_{n \rightarrow \infty} \mathbb{P} \big(G_{n,p} \mbox{ is $(H,r)$-Ramsey}\big) =
\begin{cases}
0, & \text{if $p < c n^{-1/m_2(H)}$}, \\
1, & \text{if $p > C n^{-1/m_2(H)}$},
\end{cases}
\]
where 
\[m_2(H) = \max\left\{\frac{e(H') - 1}{v(H') - 2}: H' \subseteq H \mbox{ and } v(H') \geq 3\right\}.\]
\end{theorem}

The threshold occurs at the largest value of $p^*$ such that there is some subgraph $H'$ of $H$ for which the number of copies of $H'$ is approximately the same as the number of edges. For $p$ significantly smaller than $p^*$, the number of copies of $H'$ will also be significantly smaller than the number of edges. This property allows us (by a rather long and difficult argument \cite{RR93}) to show that the edges of the graph may be colored in such a way as to avoid any monochromatic copies of $H'$. For $p$ significantly larger than $p^*$, every edge of the random graph is contained in many copies of every subgraph of $H$. The intuition, which takes substantial effort to make rigorous \cite{RR95}, is that these overlaps are enough to force the graph to be Ramsey.

Many related questions were studied in the late nineties. In particular, people were interested in determining the threshold for the following Tur\'an property. Given a graph $H$ and a real number $\e > 0$, we say that a graph $G$ is {\it $(H, \e)$-Tur\'an} if every subgraph of $G$ with at least 
\[\left(1 - \frac{1}{\chi(H) - 1} + \e\right) e(G)\]
edges contains a copy of $H$. The classical Erd\H{o}s-Stone-Simonovits theorem \cite{ESi66, ES46, T41} states that the graph $K_n$ is $(H, \e)$-Tur\'an for $n$ sufficiently large. Resolving a conjecture of Haxell, Kohayakawa, \L uczak and R\"odl \cite{HKL96, KLR97}, Conlon and Gowers \cite{CG12} and, independently, Schacht \cite{S12} proved the following theorem. It is worth noting that the result of Conlon and Gowers applies in the strictly balanced case, that is, when $m_2(H') < m_2(H)$ for all $H' \subset H$, while Schacht's result applies to all graphs. However, the class of strictly balanced graphs includes most of the graphs one would naturally consider, such as cliques or cycles.

\begin{theorem} \label{thm:randomTuran}
For any graph $H$\footnote{Note that if $H = K_2$, we take $m_2(H) = \frac{1}{2}$.} and any $\e > 0$, there exist positive
constants $c$ and $C$ such that
\[
\lim_{n \rightarrow \infty} \mathbb{P} \big(G_{n,p} \mbox{ is $(H,\e)$-Tur{\'a}n}\big) =
\begin{cases}
0, & \text{if $p < c n^{-1/m_2(H)}$}, \\
1, & \text{if $p > C n^{-1/m_2(H)}$}.
\end{cases}
\]
\end{theorem}

The results of \cite{CG12} and \cite{S12} (see also \cite{FRS10}) allow one to prove thresholds for the appearance of many different combinatorial properties. For example, the results extend without difficulty to prove analogues of Theorems \ref{thm:RR} and \ref{thm:randomTuran} for hypergraphs. The results also apply to give thresholds in different contexts - one example is an extension of Szemer\'edi's theorem to random subsets of the integers. 

Perhaps surprisingly, the methods used in \cite{CG12} and \cite{S12} are very different and have different strengths and weaknesses. We have already mentioned that Schacht's results applied to all graphs while the results of Conlon and Gowers only applied to strictly balanced graphs. On the other hand, the results of \cite{CG12} also allowed one to transfer structural statements to the sparse setting, including the stability version of the Erd\H{o}s-Stone-Simonovits theorem \cite{Si68} and the graph removal lemma. More recently, Samotij \cite{Sj12} modified Schacht's method to extend this sparse stability theorem to all graphs. The result is the following theorem.

\begin{theorem} \label{thm:randomstab}
For any graph $H$ and any $\e > 0$, there exist positive constants $\d$ and $C$ such that if $p \geq C n^{-1/m_2(H)}$ then the following holds a.a.s.~in $G_{n,p}$. Every $H$-free subgraph of $G_{n,p}$ with at least $\left(1 - \frac{1}{\chi(H) - 1} - \d\right) p \binom{n}{2}$ edges may be made $(\chi(H) - 1)$-partite by deleting at most $\e p n^2$ edges.
\end{theorem} 

Recently, a third method was developed by Balogh, Morris and Samotij \cite{BMS12} and, simultaneously and independently, by Saxton and Thomason \cite{ST12} for proving sparse random analogues of combinatorial theorems. One of the results of their research is a proof of the K\L R conjecture of Kohayakawa, \L uczak and R\"odl \cite{KLR97}. This is a technical statement which allows one to prove an embedding lemma complementing the sparse regularity lemma of Kohayakawa \cite{K97} and R\"odl. A variant of this conjecture has also been proved by Conlon, Gowers, Samotij and Schacht \cite{CGSS12} using the methods of \cite{CG12, S12}. One of the applications of this latter result is the following sparse random analogue of the graph removal lemma (this was already proved for triangles in \cite{KLR96} and for strictly balanced graphs in \cite{CG12}). 

\begin{theorem} \label{thm:removal-Gnp}
For any graph $H$ and any $\e > 0$, there exist positive constants $\d$ and $C$ such that if $p \geq Cn^{-1/m_2(H)}$ then the following holds a.a.s.~in $G_{n,p}$. Every subgraph of $G_{n,p}$ which contains at most $\d p^{e(H)} n^{v(H)}$ copies of $H$ may be made $H$-free by removing at most $\e p n^2$ edges.
\end{theorem}


Note that for any $\e$ there exists a positive constant $c$ such that if $p \leq c n^{-1/m_2(H)}$, the removal lemma is trivial. This is because, for $c$ sufficiently small, the number of copies of the densest subgraph $H'$ of $H$ will a.a.s.~be smaller than $\e p n^2$. Theorem \ref{thm:removal-Gnp} shows that it also holds for $p \geq C n^{-1/m_2(H)}$. This leaves a small intermediate range of $p$ where it might also be expected that a sparse removal lemma a.a.s.~holds. That this is so was conjectured by \L uczak \cite{L06}. 

For balanced graphs $H$, we may close the gap by letting $\delta$ be sufficiently small depending on $C, \e$ and $H$. Indeed, as $p \leq
Cn^{-1/m_2(H)}$, the number of copies of $H$ is a.a.s.~on
the order of $p^{e(H)}n^{v(H)} \leq C^{e(H)} p n^2$. Therefore, taking $\delta < \epsilon C^{-e(H)}$, we see that the number of copies of $H$ is a.a.s.~less than $\e p n^2$. Deleting one edge from each copy of $H$ in the graph then makes it H-free.

A sparse random analogue of the hypergraph removal lemma was proved in \cite{CG12} when $\mathcal{H} = K_{k+1}^{(k)}$. This result also extends to cover all strictly balanced hypergraphs.\footnote{We note that for $k$-uniform hypergraphs the relevant function is $m_k(\mathcal{H}) = \max\left\{\frac{e(\mathcal{H}') - 1}{v(\mathcal{H}') - k}\right\}$, where the maximum is taken over all subgraphs $\mathcal{H}'$ of $\mathcal{H}$ with at least $k+1$ vertices.} It would be interesting to extend this result to all hypergraphs.

It is worth noting that the sparse random version of the triangle removal lemma does not imply a sparse random version of Roth's theorem. This is because the reduction which allows us to pass from a subset of the integers with no arithmetic progressions of length $3$ to a graph containing few triangles gives us a graph with dependencies between its edges. This issue does not occur with pseudorandom graphs, which we discuss in the next section.

\subsection{Removal in pseudorandom graphs} \label{sec:pseudorandom}

Though there have long been explicit examples of graphs which behave like the random graph $G_{n,p}$, the first systematic study of what it means for a given graph to be like a random graph was initiated by Thomason \cite{Th, Th2}. Following him,\footnote{Strictly speaking, Thomason considered a slightly different notion, namely, that $|e(X) - p \binom{|X|}{2}| \leq \beta |X|$ for all $X \subseteq V$, but the two are closely related.} we say that a graph on vertex set $V$ is {\it $(p, \beta)$-jumbled} if, for all vertex subsets $X, Y \subseteq V$, 
\[|e(X,Y) - p|X||Y|| \leq \beta \sqrt{|X||Y|}.\]
The random graph $G_{n,p}$ is, with high probability, $(p, \beta)$-jumbled with $\beta = O(\sqrt{pn})$. This is also optimal in that a graph on $n$ vertices with $p \leq 1/2$ cannot be $(p, \beta)$-jumbled with $\beta = o(\sqrt{pn})$. The Paley graph is an example of an explicit graph which is optimally jumbled. This graph has vertex set $\mathbb{Z}_p$, where $p \equiv 1 (\mbox{mod } 4)$ is prime, and edge set given by connecting $x$ and $y$ if their difference is a quadratic residue. It is $(p, \beta)$-jumbled with $p = \frac{1}{2}$ and $\beta = O(\sqrt{n})$. Many more examples are given in the excellent survey \cite{KrSu}.

A fundamental result of Chung, Graham and Wilson \cite{CGW} states that for graphs of density $p$, where $p$ is a fixed positive constant, the property of being $(p, o(n))$-jumbled is equivalent to a number of other properties that one would typically expect in a random graph. For example, if the number of cycles of length $4$ is as one would expect in a binomial random graph then, surprisingly, this is enough to imply that the edges are very well-spread. 

For sparser graphs, the equivalences are less clear cut, but the notion of jumbledness defined above is a natural property to study. Given a graph property $\mathcal{P}$ that one would expect of a random graph, one can ask for the range of $p$ and $\beta$ for which a $(p, \beta)$-jumbled graph satisfies $\mathcal{P}$.

To give an example, it is known that there is a constant $c$ such that if $\beta \leq c p^2 n$ then any $(p, \beta)$-jumbled graph contains a triangle. It is also known that this is sharp, since an example of Alon \cite{A94} gives a triangle-free graph with $p = \Omega(n^{-1/3})$ which is optimally jumbled, so that $\beta = O(\sqrt{p n}) = O(p^2 n)$.

As in the previous section, one can ask for conditions on $p$ and $\beta$ which guarantee that a $(p, \beta)$-jumbled graph satisfies certain combinatorial properties. For the property of being $(K_3, \e)$-Tur\'an, this question was addressed by Sudakov, Szab\'o and Vu \cite{SSV05} (see also \cite{C05}), who showed that it was enough that $\beta \leq c p^2 n$ for an appropriate $c$. This is clearly sharp, since for larger values of $\beta$ we cannot even guarantee that the graph contains a triangle. More generally, they proved the following theorem.\footnote{Their results were only stated for the special class of $(p, \beta)$-jumbled graphs known as $(n,d,\lambda)$-graphs. These are graphs on $n$ vertices which are $d$-regular and such that all eigenvalues of the adjacency matrix, save the largest, have absolute value at most $\lambda$. The expander mixing lemma implies that these graphs are $(p, \beta)$-jumbled with $p = \frac{d}{n}$ and $\beta = \lambda$. However, it is not hard to verify that their method applies in the more general case.}

\begin{theorem} \label{thm:pseudoTuran}
For any natural number $t \geq 3$ and any $\epsilon > 0$, there exists $c > 0$ such that if $\beta \leq c p^{t-1} n$ then any $(p, \beta)$-jumbled graph is $(K_t, \epsilon)$-Tur\'an.
\end{theorem}

Except in the case of triangles, there are no known constructions which demonstrate that this theorem is tight. However, it is conjectured \cite{SSV05} that the bound on $\beta$ in Theorem \ref{thm:pseudoTuran} is the correct condition for finding copies of $K_t$ in a $(p, \beta)$-jumbled graph. This would in turn imply that Theorem~\ref{thm:pseudoTuran} is tight.

For the triangle removal lemma, the following pseudorandom analogue was recently proved by Kohayakawa, R\"odl, Schacht and Skokan \cite{KRSS10}.

\begin{theorem} \label{thm:KRSS}
For any $\e > 0$, there exist positive constants $\delta$ and $c$ such that if $\beta \leq c p^3 n$ then any $(p, \beta)$-jumbled graph $G$ on $n$ vertices has the following property. Any subgraph of $G$ containing at most $\delta p^3 n^3$ triangles may be made triangle-free by removing at most $\e p n^2$ edges.
\end{theorem}

The condition on $\beta$ in this theorem is stronger than that employed for triangles in Theorem~\ref{thm:pseudoTuran}. As a result, Alon's construction does not apply and it is an open problem to determine whether the condition $\beta \leq c p^3 n$ is optimal or if it can be improved to $\beta \leq c p^2 n$. Kohayakawa, R\"odl, Schacht and Skokan conjecture the latter, though we feel that the former is a genuine possibility.

In a recent paper, Conlon, Fox and Zhao \cite{CFZ12} found a way to prove a counting lemma for embedding any fixed small graph into a regular subgraph of a sufficiently pseudorandom host graph. Like the K\L R conjecture for random graphs, this serves to complement the sparse regularity lemma of Kohayakawa \cite{K97} and R\"odl in the pseudorandom context. As corollaries, they extended Theorems~\ref{thm:pseudoTuran} and \ref{thm:KRSS} to all graphs and proved sparse pseudorandom extensions of several other theorems, including Ramsey's theorem and the Erd\H{o}s-Simonovits stability theorem.

To state these theorems, we define the {\it degeneracy} $d(H)$ of a graph $H$ to be the smallest nonnegative integer $d$ for which there exists an ordering of the vertices of $H$ such that each vertex has at
most $d$ neighbors which appear earlier in the ordering. Equivalently, it may be defined as $d(H) = \max\{\delta(H'): H' \subseteq H\}$, where
$\delta(H)$ is the minimum degree of $H$.\footnote{In \cite{CFZ12}, a slightly different parameter, the $2$-degeneracy $d_2(H)$, is used. Though there are many cases in which this parameter is more appropriate, the degeneracy will be sufficient for the purposes of our discussion here.} 


The pseudorandom analogue of the graph removal lemma proved in \cite{CFZ12} is now as follows.\footnote{For other properties, such as that of being $(H,r)$-Ramsey or that of being $(H, \e)$-Tur\'an, an exactly analogous theorem holds with the same condition $\beta \leq c p^{d(H) + \frac{5}{2}} n$. Any of the improvements subsequently discussed for specific graphs $H$ also apply for these properties.}

\begin{theorem} \label{thm:pseudoremoval}
For any graph $H$ and any $\e > 0$, there exist positive constants $\delta$ and $c$ such that if $\beta \leq c p^{d(H) + \frac{5}{2}}n$ then any $(p, \beta)$-jumbled graph $G$ on $n$ vertices has the following property. Any subgraph of $G$ containing at most $\delta p^{e(H)} n^{v(H)}$ copies of  $H$ may be made $H$-free by removing at most $\e p n^2$ edges.
\end{theorem}

It is not hard to show, by using the random graph, that there are $(p, \beta)$-jumbled graphs with $\beta = O(p^{(d(H) + 2)/4} n)$ which contain no copies of $H$. We therefore see that the exponent of $p$ is sharp up to a multiplicative constant. However, in many cases, we expect it to be sharp up to an additive constant. 

For certain classes of graph, Theorem~\ref{thm:pseudoremoval} can be improved. For example, if we know that the degeneracy of the graph is the same as the maximum degree, such as what happens for the complete graph $K_t$, it is sufficient that $\beta \leq c p^{d(H) + 1} n$. In particular, for $K_3$, we reprove Theorem~\ref{thm:KRSS}. For cycles, the improvement is even more pronounced, since $\beta \leq c p^{t_{\ell}} n$, where $t_3 = 3$, $t_4 = 2$, $t_\ell = 1 + \frac{1}{\ell - 3}$ if $\ell \geq 5$ is odd and  $t_\ell = 1 + \frac{1}{\ell - 4}$ if $\ell \geq 6$ is even, is sufficient for removing the cycle $C_{\ell}$.

By following the proof of Kr\'al', Serra and Vena \cite{KSV09}, these bounds on the cycle removal lemma in pseudorandom graphs\footnote{Rather, a colored or directed version of this theorem.} allow us to prove an analogue of Theorem~\ref{thm:oneeqnremoval} for pseudorandom subsets of any group $G$. The {\it Cayley graph} $G(S)$ of a subset $S$ of a group
$G$ has vertex set $G$ and $(x,y)$ is an edge of $G$ if $x^{-1} y \in
S$. We say that a subset $S$ of a group $G$ is $(p,\beta)$-jumbled if
the Cayley graph $G(S)$ is $(p,\beta)$-jumbled. When $G$ is abelian, if
$\abs{\sum_{x \in S} \chi(x)} \leq \beta$ for all nontrivial
characters $\chi \colon G \to \CC$, then $S$ is $(\frac{\abs{S}}{\abs{G}},\beta)$-jumbled
(see \cite[Lemma 16]{KRSS10}). 

\begin{theorem} \label{thm:removal-groups} For any natural number $k \geq 3$ and any $\epsilon>0$, there exist positive constants $\delta$ and $c$ such that the following
  holds. Suppose $B_1,\ldots,B_k$ are subsets of a group $G$ of order $n$ such that
  each $B_i$ is $(p,\beta)$-jumbled with $\beta \leq cp^{t_k}n$. If
  subsets $A_i \subseteq B_i$ for $i=1,\ldots,k$ are such that there
  are at most $\delta |B_1|\cdots|B_k|/n$ solutions to the equation
  $x_1x_2 \cdots x_k=1$ with $x_i \in A_i$ for all $i$, then it is
  possible to remove at most $\epsilon |B_i|$ elements from each set
  $A_i$ so as to obtain sets $A_i'$ for which there are no solutions
  to $x_1x_2 \cdots x_k=1$ with $x_i \in A'_i$ for all $i$.
\end{theorem}

This result easily implies a Roth-type theorem in quite sparse pseudorandom subsets of a group. We say that a subset $B$ of a group $G$ is {\it $(\epsilon,k)$-Roth} if, for all integers $a_1,\ldots,a_k$ which satisfy $a_1+\cdots+a_k = 0$ and $\gcd(a_i,|G|)=1$ for $1 \leq i \leq k$, every subset $A \subseteq B$ which has no nontrivial solution to $x_1^{a_1}x_2^{a_2}\cdots x_k^{a_k}=1$ has $|A| \leq \epsilon |B|$.

\begin{corollary}\label{rothtype} For any natural number $k \geq 3$ and any $\epsilon>0$, there exists $c>0$ such that the following holds. If $G$ is a group of order $n$ and $B$ is a $(p,\beta)$-jumbled subset of $G$ with $\beta \leq cp^{t_k}n$, then $B$ is $(\epsilon,k)$-Roth.
\end{corollary}

Note that Roth's theorem on $3$-term arithmetic progressions in dense sets of integers follows from the special case of this result with $B=G=\mathbb{Z}_n$, $k=3$ and $a_1=a_2=1$, $a_3=-2$. The rather weak pseudorandomness condition in Corollary \ref{rothtype} shows that even quite sparse pseudorandom subsets of a group have the Roth property. 




\section{Further topics} \label{sec:topics}

\subsection{The Erd\H{o}s-Rothschild problem}

A problem of Erd\H{o}s and Rothschild \cite{E87} asks one to estimate the maximum number $h(n, c)$ such that every $n$-vertex graph with at least $c n^2$ edges, each of which is contained in at least one triangle, must contain an edge that is in at least $h(n, c)$ edges. Here, and throughout this subsection, we assume $c>0$ is a fixed absolute constant. The fact that $h(n, c)$ tends to infinity already follows from the triangle removal lemma.\footnote{Even the statement that $h(n, c) > 1$ is already enough to imply Roth's theorem.} 

To see this, suppose that $G$ is an $n$-vertex graph with $c n^2$ edges such that every edge is in at least one and at most $h := h(n,c)$ triangles. The total number of triangles in $G$ is at most $h c n^2/3$. Therefore, if $h$ does not tend to infinity, the triangle removal lemma tells us that there is a collection $E$ of $o(n^2)$ edges such that every triangle contains at least one of them. Since each edge in $G$ is in at least one triangle, we know that there are at least $c n^2/3$ triangles. It follows that some edge in $E$ is contained in at least $\omega(1)$ edges.

Using Fox's bound \cite{F11} for the triangle removal lemma, this implies that $h(n,c) \geq e^{a \log^* n}$, where $\log^* n$ is the iterated logarithm. This is defined by $\log^* x = 0$ if $x \leq 1$ and $\log^* x = \log^* (\log x) + 1$ otherwise. This improves on the bound $h(n, c) \geq (\log^* n)^a$ which follows from Ruzsa and Szemer\'edi's original proof of the triangle removal lemma.

On the other hand, Alon and Trotter (see \cite{E92}) showed that for any positive $c < \frac{1}{4}$ there is $c' > 0$ such that $h(n,c) < c' \sqrt{n}$. The condition $c < \frac{1}{4}$ is easily seen to be best possible since any $n$-vertex graph with more than $n^2/4$ edges contains an edge in at least $n/6$ triangles \cite{Ed77, KN79}. Erd\H{o}s conjectured that perhaps this behaviour is correct. That is, that for any positive $c < \frac{1}{4}$ there exists $\e > 0$ such that $h(n, c) > n^{\e}$ for all sufficiently large $n$. This was recently disproved by Fox and Loh \cite{FL12} as follows.

\begin{theorem} \label{thm:FoxLoh}
For $n$ sufficiently large, there is an $n$-vertex graph with $\frac{n^2}{4} (1 - e^{-(\log n)^{1/6}})$ edges such that every edge is in a triangle and no edge is in more than $n^{14/\log\log n}$ triangles.
\end{theorem} 

To give some idea of the construction, consider a tripartite graph between sets $A$, $B$ and $C$, each of which is a copy of a lattice cube with appropriate sidelength $r$ and dimension $d$. We join points in $A$ and $B$ if their distance is close to the expected distance between random points in $A$ and $B$. By concentration, this implies that the density of edges between $A$ and $B$ is close to $1$. We join points in $C$ to points in $A$ or $B$ if their distance is close to half the expected distance. It is not hard to see that every edge between $A$ and $B$ is then contained in few triangles. At the same time, every edge will be in at least one triangle, as can be seen by considering the midpoint of any two connected points $a$ and $b$. This yields a construction with roughly $\frac{n^2}{9}$ edges but the result of Fox and Loh may be obtained by shrinking the vertex set $C$ (or blowing up $A$ and $B$) in an appropriate fashion.

\subsection{Induced matchings}

Call a graph $G=(V,E)$ an {\it $(r,t)$-Ruzsa-Szemer\'edi graph} ($(r,t)$-RS graph for short) if its edge set can be partitioned into $t$ induced matchings in $G$, each of size $r$. The total number of edges of such a graph is $rt$. 
The most interesting problem concerns the existence of such graphs when $r$ and $t$ are both relatively large as a function of the number of vertices. The construction of Ruzsa and Szemer\'edi \cite{RuSz} using Behrend's construction demonstrates that such a graph on $n$ vertices exists with $r= e^{-c \sqrt{\log n}}n$ and $t=n/3$. The Ruzsa-Szemer\'edi result on the $(6,3)$-problem is equivalent to showing that no $(r,t)$-RS graph on $n$ vertices  exists with $r$ and $t$ linear in $n$. 

For $r$ linear in the number $n$ of vertices, it is still an open problem if there exists an $(r,t)$-RS graph with $t=n^{\epsilon}$. The best known construction in this case, due to Fischer et al. \cite{FNRR}, is an example with $r=n/3$ and $t=n^{c/\log \log n}$. However, for $r=n^{1-o(1)}$, substantial progress was made recently by Alon, Moitra and Sudakov \cite{AMS} by extending ideas used in the construction of Fox and Loh \cite{FL12} discussed in the previous subsection. They give a construction of $n$-vertex graphs with $rt=(1-o(1){n \choose 2}$ and $r=n^{1-o(1)}$. That is, there are nearly complete graphs, with edge density $1-o(1)$, such that its edge set can be partitioned into large induced matchings, each of order $n^{1-o(1)}$. They give several applications of this construction to combinatorics, complexity theory and information theory.   

\subsection{Testing small graphs}

A {\it property} of graphs is a family of graphs closed under isomorphism. A graph $G$ on $n$ vertices is {\it $\epsilon$-far} 
from satisfying a property $P$ if no graph which can be constructed from $G$ by adding and/or removing at most $\epsilon n^2$ edges satisfies $P$. 
An {\it $\epsilon$-tester} for $P$ is a randomized algorithm which, given the quantity
$n$ and the ability to make queries whether a desired pair of vertices spans an edge in $G$, distinguishes with probability at least $2/3$ 
between the case that $G$ satisfies $P$ and the case that $G$ is $\epsilon$-far from satisfying $P$. 
Such an $\epsilon$-tester is a {\it one-sided $\epsilon$-tester} if when $G$ satisfies $P$ the $\epsilon$-tester determines that this is the case. The property $P$ is called {\it testable} if, for every fixed $\epsilon > 0$, there exists a one-sided $\epsilon$-tester for $P$ whose total number of queries is bounded only by a function of $\epsilon$ 
which is independent of the size of the input graph. This means that the running time of the algorithm is also bounded by a function of
$\epsilon$ only and is independent of the input size. We measure query-complexity by the number of vertices sampled, 
assuming we always examine all edges spanned by them. The infinite removal lemma, Theorem \ref{infiniteremoval}, of Alon and Shapira \cite{AlSh08} shows that every hereditary graph property, 
that is, a graph property closed under taking induced subgraphs, is testable. Many of the best studied graph properties are hereditary. 

If the query complexity of an $\epsilon$-tester is polynomial in $\epsilon^{-1}$, we say that the property is {\it easily testable}. It is an interesting open problem to characterize the 
easily testable hereditary properties. Alon \cite{A02} considered the case where $P=P_{H}$ is the property that the graph does not contain $H$ as a subgraph. He showed that $P_H$ is easily testable if and only if $H$ is bipartite. 
Alon and Shapira \cite{AlSh06} considered the case where $P=P^*_{H}$ is the property that the graph does not contain $H$ as an {\it induced} subgraph. 
They showed that for any graph $H$ except for the path with at most four vertices, the cycle of length four and their complements, the property $P_H^*$ is not easily testable. The problem of determining whether the property $P_H^*$ is easily testable for the path with four vertices or the cycle of length four (or equivalently its complement) was left open. The case where $H$ is a path with four vertices was recently shown to be easily testable by Alon and Fox \cite{AF12}. The case where $H$ is a cycle of length four is still open. Alon and Fox also showed that if $P$ is the family of perfect graphs, then $P$ is not easily testable and, in a certain sense, testing for $P$ is at least as hard as testing triangle-freeness. 

\subsection{Local repairability}

The standard proof of the regularity lemma contains a procedure for turning a graph which is almost triangle-free into a graph which is triangle-free. We simply delete the edges between all vertex sets of low density and between all vertex sets which do not form a regular pair. This procedure can be made more explicit still by using an algorithmic version of the regularity lemma \cite{ADLRY94}. 

A surprising observation of Austin and Tao \cite{AT10} is that this repair procedure can be determined in a local fashion. They show that for any graph $H$ and any $\e > 0$ there exists $\d > 0$ and a natural number $m$ such that if $G$ is a graph containing at most $\d n^{v(H)}$ copies of $H$ then there exists a set $A$ of size at most $m$ such that $G$ may be made $H$-free by removing at most $\e n^2$ edges and the decision of whether to delete a given edge $uv$ may be determined solely by considering the restriction of $G$ to the set $A \cup \{u, v\}$.\footnote{Strictly speaking, Austin and Tao \cite{AT10} consider two forms of local repairability. Here we are considering only the weak version.}

The key point, first observed by Ishigami \cite{Ish}, is that the regular partition can be determined in a local fashion by randomly selecting vertex neighborhoods to create the partition. Since a finite set of points determine the partition, this may in turn be used to create a local modification rule which results in an $H$-free graph.

Similar ideas may also be applied to show that any hereditary graph property, including the property of being induced $H$-free, is locally repairable in the same sense. This again follows from the observation that random neighborhoods can be used to construct the partitions arising in the strong regularity lemma. 

Surprisingly, Austin and Tao show that, even though all hereditary hypergraph properties are testable, there are hereditary properties which are not locally repairable. On the other hand, they show that many natural hypergraph properties, including the property of being $\mathcal{H}$-free, are locally repairable. 


\subsection{Linear hypergraphs}

A {\it linear hypergraph} is a hypergraph where any pair of edges overlap in at most one vertex. For this special class of hypergraphs, it is not necessary to apply the full strength of hypergraph regularity to prove a corresponding removal lemma \cite{KNRS10}. Instead, a straightforward analogue of the usual regularity lemma is sufficient. This results in bounds for $\delta^{-1}$ in the linear hypergraph removal lemma which are of tower-type in a power of $\epsilon^{-1}$. 

While this is already a substantial improvement on general hypergraphs, where the best known bounds are Ackermannian,\footnote{We have already seen two levels of the Ackermann function, the tower function and the wowzer function. Generally, the $k$th level is defined by $A_k(1) = 2$ and $A_k(i+1) = A_{k-1}(A_k(i))$. Taking $A_1(i) = 2^i$, we see that $A_2(i) = T(i)$ and $A_3(i) = W(i)$. The upper bound on $\delta^{-1}$ in the $k$-uniform hypergraph removal lemma  given by the hypergraph regularity proofs are of the form $A_k(\epsilon^{-O(1)})$ or worse.} it can be improved further by using the ideas of \cite{F11}. This results in a bound of the form $T(a_{\mathcal{H}} \log \e^{-1})$.

A similar reduction does not exist for induced removal of linear hypergraphs. Because we need to consider all edges, whether present or not, between the vertices of the hypergraph, we must apply the full strength of the strong hypergraph regularity lemma. This results in Ackermannian bounds.

It is plausible that an extension of the methods of Section \ref{sec:improvedremoval} could be used to give a primitive recursive, or even tower-type, bound for hypergraph removal. We believe that such an improvement would be of great interest, not least because it would give the first primitive recursive bound for the multidimensional extension of Szemer\'edi's theorem. Such an improvement would also be likely to lead to an analogous improvement of the bounds for induced hypergraph removal.

\vspace{1mm}
{\bf Acknowledgements.} The authors would like to thank Noga Alon, Zoltan F\"uredi, Vojta R\"odl and Terry Tao for helpful comments regarding the history of the removal lemma.

\end{document}